\documentclass[12pt]{amsart}

\usepackage[margin=1in]{geometry}

\usepackage{amsmath,amssymb,amsthm}
\usepackage{color, tikz}
\usepackage{tabularx}
\usepackage{float}
\usepackage[linesnumbered,lined,boxed,commentsnumbered]{algorithm2e}
\usepackage{textcomp}
\usepackage{latexsym}
\usepackage{tikz,graphicx}
\usepackage{polynom}
\usepackage{tikz-cd}

\usepackage{rotating}
\tikzstyle{vertex}=[circle, draw, inner sep=0pt, minimum size=4pt]
\newcommand{\vertex}{\node[vertex]}

\theoremstyle{plain}
\newtheorem{theorem}{Theorem}[section]
\newtheorem{prop}[theorem]{Proposition}

\newtheorem{conjecture}[theorem]{Conjecture}

\newtheorem{lemma}[theorem]{Lemma}
\newtheorem{corollary}[theorem]{Corollary}

\theoremstyle{definition}
\newtheorem{remark}[theorem]{Remark}
\newtheorem{definition}[theorem]{Definition}

\newtheorem{example}[theorem]{Example}

\newcommand{\conv}[1]{\mathrm{conv}\{#1\}}
\newcommand{\R}{\mathbb{R}}
\newcommand{\N}{\mathbb{N}}
\newcommand{\Z}{\mathbb{Z}}

\newcommand{\Ds}{\displaystyle}

\newcommand{\maxf}{\mathrm{maxf}}
\newcommand{\minf}{\mathrm{minf}}
\newcommand{\maxfac}{\mathrm{Maxf}}
\newcommand{\minfac}{\mathrm{Minf}}

\newcommand{\PG}{P_G}
\newcommand{\PH}{P_H}
\newcommand{\facets}{N}
\newcommand{\CB}{CB}
\newcommand{\windmill}{WM}

\begin{document}

\title{Facets of Symmetric Edge Polytopes for Graphs with Few Edges}

\author{Benjamin Braun}
\address{Department of Mathematics\\
         University of Kentucky\\
         Lexington, KY 40506--0027}
\email{benjamin.braun@uky.edu}

\author{Kaitlin Bruegge}
\address{Department of Mathematics\\
         University of Kentucky\\
         Lexington, KY 40506--0027}
\email{kaitlin\_bruegge@uky.edu}

\date{5 July 2023}

\thanks{
The authors were partially supported by National Science Foundation award DMS-1953785.
The authors thank Rob Davis and Tianran Chen for helpful discussions.
The authors thank the anonymous referees for many helpful suggestions and references.}

\begin{abstract}
Symmetric edge polytopes, also called adjacency polytopes, are lattice polytopes determined by simple undirected graphs.
We introduce the integer array giving the maximum number of facets of a symmetric edge polytope for a connected graph having \(n\) vertices and \(m\) edges, and the corresponding sequence of minimal values.
We establish formulas for the number of facets obtained in several classes of sparse graphs and provide partial progress toward conjectures that identify facet-maximizing graphs in these classes.
These formulas are combinatorial in nature and lead to independently interesting observations and conjectures regarding integer sequences defined by sums of products of binomial coefficients.
\end{abstract}

\maketitle

\section{Introduction}\label{sec:intro}

Given a finite graph \(G\), there are many ways to construct a lattice polytope using \(G\) as input: graphical zonotopes, edge polytopes, matching polytopes, stable set polytopes, Laplacian simplices, flow polytopes, and others.
Of recent interest is the \emph{symmetric edge polytope} \(P_G\), introduced by Matsui, Higashitani, Nagazawa, Ohsugi, and Hibi~\cite{matsuietal2011}.
These are known as \emph{adjacency polytopes} in some applied settings~\cite{chen2017counting}.
Symmetric edge polytopes are of interest in several areas, including the study of Ehrhart theory and applications to algebraic Kuramoto equations, and these polytopes have been the subject of intense recent study~\cite{chen2021facets,chen2017counting,chen2020graphadjacency,dalidelucchimichalek,dali2022gammavector,higashitanijochemkomateusz,kalman2022ehrhart,matsuietal2011,smoothfanoehrhart2012,symmetricedgematchingpolys}.

In this paper, we study the number of facets of \(P_G\) for connected graphs, with an emphasis on those graphs having few edges.
Our study is motivated by the following question: for a fixed number of vertices and edges, what properties of connected graphs lead to symmetric edge polytopes with either a large or small number of facets?
This leads us to the following definition.

\begin{definition}\label{def:MF}
For \(n\geq 2\) and \(m\geq n-1\), define \(\maxf(n,m)\) to be the maximum number of facets of a symmetric edge polytope for a connected graph having \(n\) vertices and \(m\) edges, and similarly define \(\minf(n,m)\) to be the minimum number of facets.
For \(n\geq 2\), we define \(\maxfac(n)\) to be the maximum number of facets of a symmetric edge polytope for a connected graph having \(n\) vertices, and similarly define \(\minfac(n)\) to be the minimum number.
\end{definition}

The first few values of \(\maxf(n,m)\), sequence A360408 in OEIS~\cite{oeis}, are given in Table~\ref{tab:maxf}.
The first few values of \(\minf(n,m)\), sequence A360409 in OEIS~\cite{oeis}, are given in Table~\ref{tab:minf}.
The sequence $\maxfac(n)$ is given by 
\[
2, 6, 14, 36, 84, 216, 504, 1296,\ldots \, 
\]
while the sequence $\minfac(n)$ is given by
\[
2, 4, 6, 10, 14, 22, 30, 46,\ldots
\]

The problem of determining $\maxf(n,m)$ and $\minf(n,m)$ is challenging, in part due to the complicated combinatorial structures that describe the facets of \(P_G\).
Our experimental data suggest that facet-maximizing graphs can be obtained as wedges of odd cycles; how broadly this holds for general $n$ and $m$ beyond relatively sparse graphs is not clear.
Based on computational evidence obtained with SageMath~\cite{sage}, we offer the following conjecture regarding terms of the sequences $\maxfac(n)$ and $\minfac(n)$ in general (all undefined terms below are defined in subsequent sections).
Note that the conjectured sequence for $\minfac(n)$ is entry A027383 in OEIS~\cite{oeisA027383}.

\begin{table}[t]
\begin{tabular}{c|cccccccccccccccc}
$n$, $m-n+1$ & 0  & 1 & 2 & 3 & 4 & 5 & 6 & 7 & 8 & 9 & 10 & 11 & 12 & 13 \\
\hline 
2 & 2  \\
3 &  4 & 6 \\
4 &  8 & 12 & 12 & 14 \\
5 &  16 & 30 & 36 & 28 & 28 & 28 & 30 \\
6 &  32 & 60 & 72 & 72 & 84 & 68 & 68 & 60 & 60 & 60 & 62 \\
7 &  64 & 140 & 180 & 216 & 168 & 168 & 196 & 180 & 148 & 148 & 132 & 132 & 124 & 124 \\
\end{tabular}
\caption{$\maxf(n,m)$.}
\label{tab:maxf}
\end{table}

\begin{table}[t]
\begin{tabular}{c|cccccccccccccccc}
$n$, $m-n+1$ & 0  & 1 & 2 & 3 & 4 & 5 & 6 & 7 & 8 & 9 & 10 & 11 & 12 & 13 & 14 & 15\\
\hline 
2 & 2  \\
3 &  4 & 6 \\
4 &  8 & 6 & 12 & 14 \\
5 &  16 & 12 & 10 & 22 & 26 & 28 & 30 \\
6 &  32 & 20 & 18 & 16 & 14 & 42 & 54 & 56 & 58 & 60 & 62 \\
7 &  64 & 40 & 32 & 28 & 26 & 24 & 22 & 78 & 102 & 106 & 116 & 118 & 120 & 122 & 124 & 126\\
\end{tabular}
\caption{$\minf(n,m)$.}
\label{tab:minf}
\end{table}

\begin{conjecture}\label{conj:maxmin}
    Let $n\geq 3$.
    \begin{enumerate}
        \item For $n=2k+1$, $\maxfac(n)=6^k$, which is attained by a wedge of $k$ cycles of length three.
        \item For $n=2k$, $\maxfac(n)=14\cdot 6^{k-2}$, which is attained by a wedge of $K_4$ with $k-2$ cycles of length three.
        \item For $n=2k+1$, $\minfac(n)=3\cdot2^k-2$, which is attained by $K_{k,k+1}$.
        \item For $n=2k$, $\minfac(n)=2^{k+1}-2$, which is attained by $K_{k,k}$.
    \end{enumerate}
\end{conjecture}

The fact that the conjectured max and min values in parts~(1) and~(2) of Conjecture~\ref{conj:maxmin} are attained by a wedge follows from Proposition~\ref{prop:wedgesmultiply} below, while the analogous values for bipartite graphs in parts~(3) and~(4) were established by Higashitani, Jochemko, and Micha\l ek~\cite{higashitanijochemkomateusz}.

It is known that the symmetric edge polytope for any tree on $n$ vertices is combinatorially a cross polytope and thus has $2^{n-1}$ facets, hence $\maxf(n,n-1)=2^{n-1}$.
More generally, the number of facets for symmetric edge polytopes can be derived using combinatorial tools.
Specifically, a combinatorial description of the facet-defining hyperplanes of \(P_G\) was given by  Higashitani, Jochemko, and Micha\l ek~\cite{higashitanijochemkomateusz}.
Further, Chen, Davis, and Korchevskaia~\cite{chen2021facets} give a combinatorial description of the faces of \(P_G\) that utilizes special subgraphs of \(G\).

It follows from Definition~\ref{def:pg} below that symmetric edge polytopes are centrally symmetric lattice polytopes.
Symmetric edges polytopes have also been shown to be reflexive and terminal~\cite{higashitanifanopolytopes}.
Further, Higashitani~\cite[Theorem 3.3]{higashitanifanopolytopes} proved that centrally symmetric simplicial reflexive polytopes are precisely the symmetric edge polytopes of graphs without even cycles.
In Conjecture~\ref{conj:maxmin}(1), the symmetric edge polytopes arising from wedges of cycles of length three fall within this family.
This is related to a result due to Nill~\cite[Corollary 4.4]{NillClassification} stating that the maximum number of facets for any pseudo-symmetric reflexive simplicial \(d\)-polytope \(P\) is \(6^{d/2}\) and that the maximum is attained if and only if \(P\) is a free sum of \(d/2\) copies of \(P_{K_3}\).
Thus, Conjecture~\ref{conj:maxmin}(1) aligns with existing results regarding these polytopes.

In this work, we investigate the sequences $\maxf(n,n)$ and $\maxf(n,n+1)$.
We provide an exact result for $\maxf(n,n)$ and provide partial progress toward a conjectured value of $\maxf(n,n+1)$.
The use of combinatorial tools for this analysis  produces independently interesting integer sequences defined by sums of products of binomial coefficients.

This paper is structured as follows.
In Section~\ref{sec:background}, we provide necessary definitions and background.  
In Section~\ref{sec:disjoint_cycle}, we give formulas for the number of facets and discuss facet-maximizers among some sparse connected graphs, namely graphs on \(n\) vertices with \(n\) or \(n+1\) edges where any cycles present are edge-disjoint.
In these cases, Theorems~\ref{thm:nnmax}~and~\ref{thm: disjoint cycles less than M(n)} respectively describe facet-maximizing graphs. 
In Section~\ref{sec:overlap_cycle}, we discuss facet counts for graphs constructed from internally disjoint paths connected at their endpoints and give formulas in Propositions~\ref{prop: same parity}~and~\ref{prop:facets of gen cb mixed}.
As a special case of this, we get results about the number of facets arising from graphs with \(n\) vertices and \(n+1\) edges where the cycles share at least one edge, and we make progress toward generalizing Theorem~\ref{thm: disjoint cycles less than M(n)} to this class of graphs. 
In Section~\ref{sec:conjectures}, we give several conjectures regarding facet-maximizing graphs in certain families.  
We also discuss computational evidence supporting these conjectures.

\section{Background}\label{sec:background}

\begin{definition}\label{def:pg}
Let \(G\) be a graph on the vertex set \([n]=\{1,\ldots,n\}\) and edge set \(E=E(G)\).
Let \(e_i\) denote the \(i\)-th standard basis vector in \(\R^n\) and let \(\conv{X}\) denote the convex hull of a subset \(X\subset \R^n\).
The \emph{symmetric edge polytope} for $G$ is
\[
\PG:=\conv{\pm(e_i-e_j): \{i,j\}\in E(G)}\,.
\]

We denote by \(N(P)\) the number of facets of a polytope \(P\).
We denote by both \(N(\PG)\) and \(N(G)\) the number of facets of \(\PG\).
\end{definition}

\begin{example}
Let \(G\) be the path with vertices \(\{1,2,3\}\) and edges \(\{12,23\}\).
Then 
\[
\PG=\conv{\pm(e_1-e_2),\pm(e_2-e_3)}\subset \R^3
\]
is a \(4\)-gon contained in the orthogonal complement of the vector \(\langle 1,1,1\rangle\). 
This polygon has four 1-dimensional faces.
So \(\facets(\PG)=4\).
\end{example}

In general, the machinery used to count the facets of $\PG$ are functions \(f:V\rightarrow \Z\) on the set \(V\) of vertices in \(G\) satisfying certain properties. 
It was shown in~\cite[Theorem 3.1]{higashitanijochemkomateusz} that the facets of \(\PG\) are in bijection with these functions. 

\begin{theorem}[Higashitani, Jochemko, Micha\l ek \cite{higashitanijochemkomateusz}]\label{thm:facetdescription}
Let \(G=(V,E)\) be a finite simple connected graph.
Then \(f:V\rightarrow\Z\) is facet-defining if and only if both of the following hold.
\begin{enumerate}
    \item[(i)]For any edge \(e=uv\) we have \(|f(u)-f(v)|\leq 1\).
    \item[(ii)] The subset of edges \(E_f=\{e=uv\in E\::\: |f(u)-f(v)|= 1\}\) forms a spanning connected subgraph of \(G\).
\end{enumerate}
\end{theorem}

As symmetric edge polytopes are contained in the hyperplane orthogonal to the span of the vector where every entry is one, two facet-defining functions are identified if they differ by a common constant.
The spanning connected subgraphs with edge sets \(E_f\) arising in Theorem~\ref{thm:facetdescription}, called \emph{facet subgraphs}, have further structure. 

\begin{lemma}[Chen, Davis, Korchevskaia \cite{chen2021facets}]\label{lem:facetsubgraphs}
Let \(G\) be a connected graph. 
A subgraph \(H\) of \(G\) is a facet subgraph of \(G\) if and only if it is a maximal connected spanning bipartite subgraph of \(G\).
\end{lemma}

Lemma~\ref{lem:facetsubgraphs} provides a strategy for identifying the facets of \(P_G\) combinatorially: first identify the maximal connected spanning bipartite subgraphs of \(G\), then determine the valid integer labelings of the vertices.
Facet counts for symmetric edge polytopes are known for certain classes of graphs. 
A class of particular interest to us is cycles.
Let \(C_n\) denote the cycle with \(n\) edges and let \(Q_n\) denote the path with \(n\) edges.

\begin{lemma}\label{lem:cyclefacets}
For any \(m\),
\[
\facets(P_{C_m})=
\left\{
\begin{array}{cl}
\binom{m}{m/2} & m\text{ even}\\
m\binom{m-1}{(m-1)/2} &  m\text{ odd}
\end{array}
\right.
\]
\end{lemma}
\begin{proof}
For even \(m\), the facets of \(P_{C_m}\) are identified and counted in~\cite[Proposition~12]{chen2017counting}, and for odd \(m\) in~\cite[Remark~4.3]{nill2006}.
\end{proof}

Though the two-cycle, \(C_2\), is a multigraph (and thus, its symmetric edge polytope is not defined), its facet-defining functions would be exactly the facet-defining functions of a graph on two vertices with a single edge.  
This is consistent with the formula in Lemma~\ref{lem:cyclefacets}.

For a graph \(G\) that is constructed by identifying two graphs at a single vertex, there is a relationship between the facets of \(\PG\) and the facets of the subgraphs.  

\begin{definition}
For graphs \(G\) and \(H\), let \(G\vee H\) denote a graph obtained by identifying a vertex in \(G\) with a vertex in \(H\).
We call $G\vee H$ a \emph{wedge} or \emph{join}.
\end{definition}

Note that we do not specify a choice of identification points when defining \(G\vee H\), as by the following proposition any such choice yields a symmetric edge polytope with the same number of facets.

\begin{prop}\label{prop:wedgesmultiply}
For connected graphs \(G\) and \(H\),
\[
\facets(P_{G\vee H})=\facets(\PG)\cdot\facets(P_H).
\]
\end{prop}

\begin{proof}
This follows from the fact that \(P_{G\vee H}\) is the free sum \(\PG \oplus \PH\)~\cite[Proposition~4.2]{osughitsuchiya2020} (also called the direct sum) and the number of facets is multiplicative for free sums~\cite{HenkRichterZiegler}.
\end{proof}

\section{Graphs with Few Edges and Disjoint Cycles}\label{sec:disjoint_cycle}

We consider the symmetric edge polytopes for classes of connected graphs where the number of edges is small relative to the number of vertices.
For any tree \(T\) on \(n\) vertices, \(\facets(T)=2^{n-1}\) by Proposition~\ref{prop:wedgesmultiply} as \(T\) can be constructed as a wedge of \(n-1\) single edges with an appropriate choice of identification points. 
Thus, $\maxf(n,n-1)=2^{n-1}$.

Considering next the sequence $\maxf(n,n)$, any connected graph with an equal number of vertices and edges has a unique cycle, and hence can be constructed as a wedge of that cycle with trees. 
Therefore, we can count the facets of \(\PG\) for any such graph \(G\) and determine the maximum possible facet number arising from a graph with \(n\) vertices and \(n\) edges.

\begin{definition}\label{def:cnm}
Let \(C(n,m)\) denote a graph on \(n\) vertices obtained by joining an \(m\)-cycle with a path graph on \(n-m\) edges. 
\end{definition}

\begin{figure}
    \centering
\begin{tikzpicture}
    \vertex[fill](1) at (-1,2.9) {};
	\vertex[fill] (2) at (0,2) {};
	\vertex[fill] (3) at (1,2.5) {};
	\vertex[fill](4) at (-2,2) {};
	\vertex[fill](5) at (-1.5,0.8) {};
	\vertex[fill] (6) at (-0.5,0.8) {};
	\vertex[fill] (7) at (1,1.5) {};
	\draw[thick] (1)--(2) {};
	\draw[thick] (1)--(4) {};
	\draw[thick] (4)--(5) {};
	\draw[thick] (2)--(6) {};
	\draw[thick] (5)--(6) {};
	\draw[thick] (2)--(3) {};
	\draw[thick] (3)--(7) {};
\end{tikzpicture}
    \caption{\(C(7,5)\)}
    \label{fig:cnm}
\end{figure}

\begin{theorem}\label{thm:nnmax}
For any connected graph \(H\) with \(n\) vertices and \(n\) edges, the number of facets of \(P_H\) is less than or equal to the number of facets of \(\PG\) for \(G=C(n,n)\) when \(n\) is odd and \(G=C(n,n-1)\) when \(n\) is even. 
Thus, for odd $n$ 
\[
\maxf(n,n)=n\binom{n-1}{(n-1)/2} \, ,
\]
and for even $n$ 
\[
\maxf(n,n)= 2(n-1)\binom{n-2}{(n-2)/2}\, .
\]
\end{theorem}

\begin{proof}
A connected graph on \(n\) vertices and \(n\) edges has a unique cycle of length \(m\) for some \(3\leq m\leq n\).
Thus, \(G\) is the join of an \(m\)-cycle and \(n-m\) edges. 
By Proposition~\ref{prop:wedgesmultiply}, we have $\facets(G)=\facets(C(n,m))$.
For \(k\geq 2\), we claim 
\begin{equation}\label{eqn:cniineq}
\facets(C(n,2k))<\facets(C(n,2k-1))<\facets(C(n,2k+1))\, .
\end{equation}
In other words, if \(m\) is even, \(\facets(C(n,m-1))\) is greater than \(\facets(C(n,m))\). 
Also, if \(m\) is odd and \(m\leq n-2\), the graph \(C(n,m+2)\) exists, and \(\facets(C(n,m+2))\) is greater than \(\facets(C(n,m))\).  
With these two statements, we see that \(\facets(G)\) is maximized when \(G\) contains the largest odd cycle possible in a graph with \(n\) vertices.

To prove the inequality in~\eqref{eqn:cniineq}, let 
\[
\mathcal{M}=\frac{2^{n-(2k+1)}(2k-1)!}{(k!)^2}.
\]
Then, by Lemma~\ref{lem:cyclefacets} and Proposition~\ref{prop:wedgesmultiply}, 
\[
\facets(C(n,2k))=\facets(C_{2k})\cdot 2^{n-2k}=\binom{2k}{k}\cdot 2^{n-2k} = 4k\mathcal{M},
\]
\[
\facets(C(n,2k-1))=\facets(C_{2k-1})\cdot 2^{n-(2k-1)}=(2k-1)\binom{2k-2}{k-1}\cdot 2^{n-(2k-1)}=4k^2\mathcal{M},
\]
\[
\facets(C(n,2k+1))=\facets(C_{2k+1})\cdot 2^{n-(2k+1)}=(2k+1)\binom{2k}{k}\cdot 2^{n-(2k+1)}=(4k^2+2k)\mathcal{M},
\]
and the claim holds.
\end{proof}

We next consider the sequence $\maxf(n,n+1)$, which is substantially more challenging than the previous cases.
Any connected, simple graph with \(n\) vertices and \(n+1\) edges can be constructed from a tree on \(n\) vertices by adding two edges.  
Each of these additions induces a cycle in the graph.
For such graphs, we make the following definition and conjecture.

\begin{definition}\label{def:M(n)}
For \(n\geq 3\), let \(M(n)\) be the number of facets of \(\PG\) where 
\[
G:=\begin{cases}
C_{k+1}\vee C_{k-1}&n=2k-1,\text{ \(k\) even}\\
C_{k}\vee C_{k}  &n=2k-1,\text{ \(k\) odd}\\
C_{k+1}\vee C_{k-1}\vee e\;\;&n=2k,\text{ \(k\) even}\\
C_{k}\vee C_{k}\vee e &n=2k,\text{ \(k\) odd}
\end{cases}
\]
\end{definition}

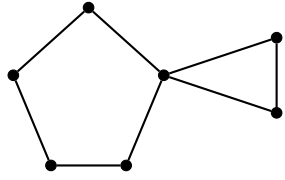
\begin{figure}
    \centering
\begin{tikzpicture}
    \vertex[fill](1) at (-1,2.9) {};
	\vertex[fill] (2) at (0,2) {};
	\vertex[fill] (3) at (1.5,2.5) {};
	\vertex[fill](4) at (-2,2) {};
	\vertex[fill](5) at (-1.5,0.8) {};
	\vertex[fill] (6) at (-0.5,0.8) {};
	\vertex[fill] (7) at (1.5,1.5) {};
	\draw[thick] (1)--(2) {};
	\draw[thick] (1)--(4) {};
	\draw[thick] (4)--(5) {};
	\draw[thick] (2)--(6) {};
	\draw[thick] (5)--(6) {};
	\draw[thick] (2)--(3) {};
	\draw[thick] (2)--(7) {};
	\draw[thick] (3)--(7) {};
\end{tikzpicture}
    \caption{A graph with \(\facets(G)=M(7)\)}
    \label{fig:M(7)}
\end{figure}

\begin{conjecture}\label{conj:nn+1}
For all $n\geq 3$, $\maxf(n,n+1)$ is equal to $M(n)$.
\end{conjecture}

Graphs with $n$ vertices and $n+1$ edges fall into two categories: graphs with exactly 2 edge-disjoint cycles, such as those defined in Definition~\ref{def:two cycle wedge graphs} below, and graphs where the cycles share one or more edges, such as those defined in Definition~\ref{def:gen CB} below.
In this section, we show that Conjecture~\ref{conj:nn+1} is true for the first category.

\begin{theorem}\label{thm: disjoint cycles less than M(n)}
For any connected graph \(H\) with \(n\) vertices and \(n+1\) edges where \(H\) contains two edge-disjoint cycles, we have \(\facets(H)\leq M(n)\).
\end{theorem}

Note that Theorem~\ref{thm: disjoint cycles less than M(n)} states that, among connected graphs with \(n\) vertices and \(n+1\) edges containing disjoint cycles, a facet-maximizing family arises by creating a graph that as closely as possible resembles the wedge of two equal-length odd cycles. 
The proof relies on the following definition and lemmas.

\begin{definition}\label{def:two cycle wedge graphs}
Let \(G(n,i,j)\) denote the graph \(C_i\vee C_j\vee Q_{n+1-(i+j)}\).
Note that \(G(n,i,j)\) has \(n\) vertices and \(n+1\) edges.
\end{definition}

\begin{lemma}\label{lem:odd cycle better than even} 
If \(i\) is even, then 
\[
\facets(G(n,i,j)) < \facets(G(n,i-1,j)).
\] 
\end{lemma}

\begin{proof}
Note that \(\facets(G(n,i,j))=\facets(C(n+1-j, i)\vee C_j)\).
Because \(i\) is even, applying~\eqref{eqn:cniineq} and Proposition~\ref{prop:wedgesmultiply} yields
\[
\facets(C(n+1-j, i)\vee C_j)<\facets(C(n+1-j, i-1)\vee C_j)=\facets(G(n,i-1,j))\, ,
\]
which completes the proof.
\end{proof}

\begin{lemma}\label{lem:evenly distributed cycles}
For \(i,j,m,\ell\) odd with \(m < i\leq j<l\) and \(i+j=m+\ell\),
\[
\facets(C_m\vee C_\ell) < \facets(C_i\vee C_j).
\]
\end{lemma}

\begin{proof}
We show this holds for \(m=i-2\) and \(\ell = j+2\), then the argument follows by induction.
By Lemma~\ref{lem:cyclefacets} and Proposition~\ref{prop:wedgesmultiply}
\[
\facets(C_i\vee C_j)=ij\binom{i-1}{\frac{i-1}{2}}\binom{j-1}{\frac{j-1}{2}},
\]
\[
\facets(C_{i-2}\vee C_{j+2})=(i-2)(j+2)\binom{i-3}{\frac{i-3}{2}}\binom{j+1}{\frac{j+1}{2}}.
\]

Letting \(\Ds\mathcal{M}=(i-2)j\binom{i-3}{\frac{i-3}{2}}\binom{j-1}{\frac{j-1}{2}}\), we see
\[
\facets(C_{i-2}\vee C_{j+2})=4\mathcal{M}\cdot\frac{j+2}{j+1}<4\mathcal{M}\cdot\frac{i}{i-1}=\facets(C_i\vee C_j).
\]
\end{proof}

We also make use of the following theorem.

\begin{theorem}\label{thm:M(n) inequailty}
For all \(n\) 
\[
2\cdot M(n)\leq M(n+1).
\]
\end{theorem}

\begin{proof}
By  Lemma~\ref{lem:cyclefacets} and Proposition~\ref{prop:wedgesmultiply},
\[
M(n)=\begin{cases}
(k+1)(k-1)\binom{k}{\frac{k}{2}}\binom{k-2}{\frac{k-2}{2}}\;\; &n=2k-1,\; k\text{ even}\\
k^2\binom{k-1}{\frac{k-1}{2}}^2\;\; &n=2k-1,\; k\text{ odd}\\
2\cdot M(n-1)\;\; &n=2k.
\end{cases}
\]
When \(n\) is odd, \(2\cdot M(n)= M(n+1)\), and we are done.
When \(n\) is even, we consider two cases.

\noindent\emph{Case 1:} If \(n=2k\) with \(k\) even, \(n+1=2(k+1)-1\) with \(k+1\) odd. Therefore, letting  \(\Ds\mathcal{K}=(k+1)\binom{k}{\frac{k}{2}}\), we have
\[
 M(n)=2\cdot M(2k-1)=2(k+1)(k-1)\binom{k}{\frac{k}{2}}\binom{k-2}{\frac{k-2}{2}}=2(k-1)\binom{k-2}{\frac{k-2}{2}}\cdot\mathcal{K},    
\]
and
\[
M(n+1)=(k+1)^2\binom{k}{\frac{k}{2}}^2=\mathcal{K}^2 \, .
\]

Since 
\[
\begin{split}
\mathcal{K}&=(k+1)\binom{k}{\frac{k}{2}}=\frac{(k+1)k}{\left(\frac{k}{2}\right)^2}\cdot(k-1)\cdot\binom{k-2}{\frac{k-2}{2}}\\&=\frac{4(k+1)}{k}\cdot(k-1)\cdot\binom{k-2}{\frac{k-2}{2}}\geq 4\cdot(k-1)\cdot\binom{k-2}{\frac{k-2}{2}}\, ,    
\end{split}
\]
we see
\[
2\cdot M(n)=4(k-1)\binom{k-2}{\frac{k-2}{2}}\cdot \mathcal{K}\leq \mathcal{K}^2=M(n+1)\, .    
\]

\noindent\emph{Case 2:} If \(n=2k\) with \(k\) odd, \(n+1=2(k+1)-1\) with \(k+1\) even.
Therefore, letting \(\mathcal{K}=k\binom{k-1}{\frac{k-1}{2}}\), we have
\[
M(n)=2\cdot M(2k-1)=2k^2\binom{k-1}{\frac{k-1}{2}}^2=2\mathcal{K}^2
\]
and
\[
M(n+1)=(k+2)k\binom{k+1}{\frac{k+1}{2}}\binom{k-1}{\frac{k-1}{2}}=(k+2)\binom{k+1}{\frac{k+1}{2}}\cdot\mathcal{K}\, .
\]
Since 
\[
(k+2)\binom{k+1}{\frac{k+1}{2}}=\frac{(k+2)(k+1)}{\left(\frac{k+1}{2}\right)^2}\cdot k\cdot\binom{k-1}{\frac{k-1}{2}}=\frac{4(k+2)}{(k+1)}\cdot \mathcal{K}\geq 4\mathcal{K} \, ,
\]
we see
\[
2\cdot M(n)=4\mathcal{K}^2\leq (k+2)\binom{k+1}{\frac{k+1}{2}}\cdot\mathcal{K}=M(n+1)\, . 
\]
\end{proof}

With these in place, we have Theorem~\ref{thm: disjoint cycles less than M(n)}.

\begin{proof}[Proof of Theorem~\ref{thm: disjoint cycles less than M(n)}]
By Proposition~\ref{prop:wedgesmultiply}, any such graph \(H\) containing exactly two edge-disjoint cycles of length \(i\) and \(j\) satisfies \(\facets(H)=\facets(G(n,i,j))\).
Thus, it is sufficient to restrict our attention to \(G(n,i,j)\).
By Lemmas~\ref{lem:odd cycle better than even}~and~\ref{lem:evenly distributed cycles}, we can consider only \(G(n,i,j)\) for \(i,j\) odd and as close to \(\frac{i+j}{2}\) as possible. Without loss of generality, suppose \(i\leq j\).

\noindent\emph{Case 1: \(n=2k-1\), \(k\) even.} Note that, since \(i\) and \(j\) are both odd, \(i\leq k-1\) and \(j\leq k+1\).
Also, by  Lemma~\ref{lem:cyclefacets} and Proposition~\ref{prop:wedgesmultiply}, 
\[
\facets(G(n,i,j))=2^{n-(i+j)+1}ij\binom{i-1}{\frac{i-1}{2}}\binom{j-1}{\frac{j-1}{2}}
=2^{2k-(i+j)}\cdot\frac{i!j!}{\left(\frac{i-1}{2}!\right)^2\left(\frac{j-1}{2}!\right)^2} \, .
\]
Similarly,
\[
\begin{split}
    M(n)&=(k+1)(k-1)\binom{k}{\frac{k}{2}}\binom{k-2}{\frac{k-2}{2}}
    \\&=\frac{(k+1)!}{\left(\frac{k}{2}!\right)^2}\cdot \frac{(k-1)!}{\left(\frac{k-2}{2}!\right)^2}
    \\&=\frac{j!}{\left(\frac{j-1}{2}!\right)^2}\cdot\left(\prod_{\substack{\ell =j+1\\\ell\text{ even}}}^{k} \frac{(\ell + 1)\ell}{\left(\frac{\ell}{2}\right)^2} \right)
    \cdot\frac{i!}{\left(\frac{i-1}{2}!\right)^2}\cdot\left(\prod_{\substack{m =i+1\\m\text{ even}}}^{k-2} \frac{(m + 1)m}{\left(\frac{m}{2}\right)^2} \right)
    \\&=2^{2k-(i+j)}\cdot\frac{i!j!}{\left(\frac{i-1}{2}!\right)^2\left(\frac{j-1}{2}!\right)^2}\cdot\left(\prod_{\substack{\ell =j+1\\\ell\text{ even}}}^{k} \frac{\ell + 1}{\ell}\right)\cdot\left(\prod_{\substack{m =i+1\\m\text{ even}}}^{k-2}\frac{m+1}{m}\right)
    \\&\geq\facets(G(n,i,j)).
\end{split}
\]

\noindent\emph{Case 2: \(n=2k-1\), \(k\) odd.} Note that, by assumption, \(i\leq k\) and \(j\leq k\). 
Also, by  Lemma~\ref{lem:cyclefacets} and Proposition~\ref{prop:wedgesmultiply}, 
\[
\facets(G(n,i,j))=2^{n-(i+j)+1}ij\binom{i-1}{\frac{i-1}{2}}\binom{j-1}{\frac{j-1}{2}}
=2^{2k-(i+j)}\cdot\frac{i!j!}{\left(\frac{i-1}{2}!\right)^2\left(\frac{j-1}{2}!\right)^2} \, .
\]
Similarly,
\[
\begin{split}
    M(n)&=k^2\binom{k-1}{\frac{k-1}{2}}^2
    \\&=\frac{j!}{\left(\frac{j-1}{2}!\right)^2}\cdot\left(\prod_{\substack{\ell =j+1\\\ell\text{ even}}}^{k-1} \frac{(\ell + 1)\ell}{\left(\frac{\ell}{2}\right)^2} \right)
    \cdot\frac{i!}{\left(\frac{i-1}{2}!\right)^2}\cdot\left(\prod_{\substack{m =i+1\\m\text{ even}}}^{k-1} \frac{(m + 1)m}{\left(\frac{m}{2}\right)^2} \right)
    \\&=2^{2k-(i+j)}\cdot\frac{i!j!}{\left(\frac{i-1}{2}!\right)^2\left(\frac{j-1}{2}!\right)^2}\left(\prod_{\substack{\ell =j+1\\\ell\text{ even}}}^{k-1} \frac{\ell + 1}{\ell}\right)\cdot\left(\prod_{\substack{m =i+1\\m\text{ even}}}^{k-1}\frac{m+1}{m}\right)
    \\&\geq\facets(G(n,i,j)).
\end{split}
\]

\noindent\emph{Case 3: \(n=2k\).} By Lemma~\ref{lem:cyclefacets} and Proposition~\ref{prop:wedgesmultiply}, 
\[
\facets(G(n,i,j))=2^{n-(i+j)+1}ij\binom{i-1}{\frac{i-1}{2}}\binom{j-1}{\frac{j-1}{2}} = 2\cdot\facets(G(n-1,i,j))\, .
\]
Also, by Theorem~\ref{thm:M(n) inequailty} and the previous cases,
\[
\begin{split}
    M(n)&\geq 2\cdot M(n-1)
    \\&\geq 2\cdot \facets(G(n-1,i,j))
    \\&= \facets(G(n,i,j)).
\end{split}
\]

Thus, in every case, \(\facets(G(n,i,j))\leq M(n)\).
\end{proof}

\section{Graphs with Few Edges and Overlapping Cycles}\label{sec:overlap_cycle}

We next consider the family of graphs on \(n\) vertices and \(n+1\) edges that have two cycles intersecting in at least one edge.
Note that the case where two cycles intersect in exactly one edge was studied in~\cite[Section 5]{chen2021facets}.
Theorem~\ref{thm:M(n) inequailty} allows us to reduce Conjecture~\ref{conj:nn+1} to the case where \(G\) has no vertices of degree one, as follows.

\begin{corollary}\label{cor:noleaves}
If Conjecture~\ref{conj:nn+1} is true for graphs on \(n\) vertices, then it is true for graphs on \(n+1\) vertices that have at least one leaf.
\end{corollary}

\begin{proof}
Let \(G\) be a graph on \(n+1\) vertices and \(n+2\) edges that has a leaf \(e\). 
Then \(G\setminus\{e\}\) is a graph with \(n\) vertices and \(n+1\) edges, and by assumption \(\facets(G\setminus\{e\})\leq M(n)\).
By Proposition~\ref{prop:wedgesmultiply} and Theorem~\ref{thm:M(n) inequailty},
\[
\facets(G)=2\cdot\facets(G\setminus\{e\})\leq 2\cdot M(n)\leq M(n+1).
\]
\end{proof}

Corollary~\ref{cor:noleaves} allows us to restrict our attention to graphs with no leaves. 
Any graph on \(n\) vertices and \(n+1\) edges with no leaves that contains cycles sharing one or more edges can be interpreted as three internally disjoint paths connected at their endpoints. 
We consider these as a special case of a more general construction.

\begin{definition}\label{def:gen CB}
For a vector \(\mathbf{m}\in \N^t\), let \(\CB(\mathbf{m})\) denote the graph made of \(t\) internally disjoint paths of lengths \(m_1, m_2,\dots, m_t\) connecting two endpoints.
\end{definition}

Note that when $t=3$, we obtain the leafless connected graphs with $n$ vertices and $n+1$ edges.

\begin{remark}
The graphs \(\CB(\mathbf{m})\) for which all entries of \(\mathbf{m}\) are the same \(m\in \N\) are sometimes called \emph{theta graphs}, denoted by \(\theta_{m,t}\)~\cite{thetagraphs}.
\end{remark}

\begin{prop}\label{prop: same parity}
For \(\mathbf{m}\in \N^t\), we may permute the entries so that without loss of generality we have \(m_1\geq m_2\geq\cdots\geq m_t\).  
If all the \(m_i\)'s have the same parity, \(\facets(\CB(\mathbf{m}))\) is given by 
\[
F(\mathbf{m})=\sum_{j=0}^{m_t} \binom{m_t}{j}\left[\prod_{k=1}^{t-1}\binom{m_k}{\frac{1}{2}(m_k-m_t)+j}\right].
\]
\end{prop}

\begin{proof}
Consider \(\CB(\mathbf{m})\) as consisting of paths \(Q_1,Q_2,\dots,Q_t\) having \(m_1\geq \cdots \geq m_t\) edges respectively, as shown in Figure~\ref{fig:cb gen}.
Since all \(m_i\) are the same parity, \(\CB(\mathbf{m})\) is bipartite.
For every facet-defining function \(f:V\rightarrow \Z\), we have \(|f(u)-f(v)|=1\) for every edge \(uv\) in \(\CB(\mathbf{m})\)~\cite[Lem.~4.5]{dalidelucchimichalek}. 
If we consider the paths as oriented away from the top vertex toward the bottom vertex, we can view each edge \(u\rightarrow v\) as ascending (\(f(v)-f(u)=1\)), and label it \(1\), or descending (\(f(v)-f(u)=-1\)), and label it \(-1\).

We count facets by finding \emph{valid} labelings of the edges of \(\CB(\mathbf{m})\) with \(\pm 1\), that is labelings such that the sum of the labels on every path is the same.
For a labeling of a shortest path with length \(m_t\) using \(j\) \(-1\)s and \(m_t-j\) \(1\)s, the sum of the edge labels is \(m_t-2j\).
There are \(\binom{m_t}{j}\) labelings of this path with this sum.

To produce a valid labeling of the entire graph with each path sum equal to \(m_t-2j\), the number of \(-1\)s, say \(y\), on a path of length \(m_k\) must satisfy the equation:
\[
\begin{split}
    (+1)(m_k-y)+(-1)y&=m_t-2j\\
    y&=\frac{1}{2}(m_k-m_t)+j.
\end{split}
\]
Thus there are \(\Ds\binom{m_k}{\frac{1}{2}(m_k-m_t)+j}\) labelings of a path of length \(m_k\) with label sum \(m_t-2j\).  Applying this argument to \(m_k\) for \(k=1,\dots, t-1\) gives
\[
\binom{m_t}{j}\prod_{k=1}^{t-1}\binom{m_k}{\frac{1}{2}(m_k-m_t)+j}
\]
valid labelings of \(\CB(\mathbf{m})\) with \(j\) \(-1\)s on the shortest path. The result follows by taking the sum over all \(j=0,\dots, m_t\).
\end{proof}

Note that there is a combinatorial interpretation for \(F\), where we consider the arithmetical triangle of binomial coefficients vertically centered at the central terms.
What \(F\) does is select the \(m_t\)-th row of the arithmetical triangle, multiply each entry by the vertically-aligned entries in rows \(m_1\) through \(m_{t-1}\), and sum the resulting products.

For \(\CB(\mathbf{m})=\theta_{m,t}\) where all the paths are the same length, this formula simplifies.

\begin{corollary}
For \(t\in \N\),
\[
\facets(\theta_{m,t})=\sum_{j=0}^m\binom{m}{j}^t.
\]
\end{corollary}

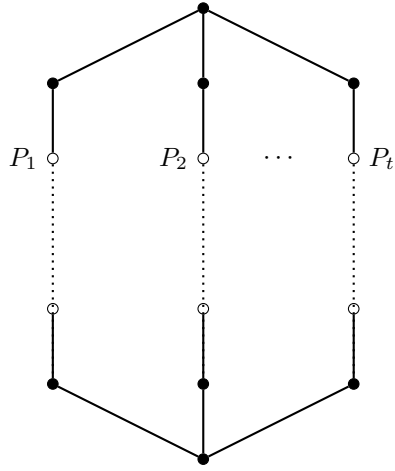
\begin{figure}
\begin{center}
\begin{tikzpicture}
\begin{scope}[scale=1, xshift=0, yshift=0]
	\vertex[fill](a1) at (0,3) {};
	\vertex[fill](a2) at (0,2) {};
	\vertex[fill](a6) at (0,-2) {};
	\vertex[fill](a7) at (0,-3) {};
    \vertex[fill](l2) at (-2,2) {};
	\vertex[fill](l6) at (-2,-2) {};
	\vertex[fill](r2) at (2,2) {};
	\vertex[fill](r6) at (2,-2) {};
	\vertex[label=left:\(Q_{1}\)] (l4) at (-2,1) {};
	\vertex (l5) at (-2,-1) {};
	\vertex[label=left:\(Q_{2}\)] (a4) at (0,1) {};
	\vertex (a5) at (0,-1) {};
	\vertex[label=right:\(Q_{t}\)] (r4) at (2,1) {};
	\vertex (r5) at (2,-1) {};
	\node[] at (1,1) {\(\cdots\)};
	\draw[thick] (a1)--(a2);
	\draw[thick] (a1)--(r2);
	\draw[thick] (a1)--(l2);
	\draw[thick] (a2)--(a4);
	\draw[thick] (r2)--(r4);
	\draw[thick] (l2)--(l4);
	\draw[thick, dotted] (a4)--(a6);
	\draw[thick, dotted] (r4)--(r6);
	\draw[thick, dotted] (l4)--(l6);
	\draw[thick] (a5)--(a6);
	\draw[thick] (r5)--(r6);
	\draw[thick] (l5)--(l6);
	\draw[thick] (a6)--(a7);
	\draw[thick] (a7)--(r6);
	\draw[thick] (a7)--(l6);
\end{scope}
\end{tikzpicture}
\end{center}
\caption{\(\CB(\mathbf{m})\) for \(\mathbf{m}=(m_1,\dots, m_t)\)}
\label{fig:cb gen}
\end{figure}

If the vector \(\mathbf{m}\) has both even and odd entries, counting the facets of \(\CB(\mathbf{m})\) becomes more complicated, but still involves \(F\).

\begin{prop}\label{prop:facets of gen cb mixed}
For \(\mathbf{m}\in \N^t\), permute the entries so that \(\mathbf{m}=(e_1,\dots, e_k,o_1,\dots, o_\ell)\) with \(e_1\geq e_2\geq\cdots\geq e_k\) even and \(o_1\geq o_2\geq\cdots\geq o_\ell\) odd and \(k,\ell\geq 1\), \(k+\ell=t\). 
Also, let \(\mathbf{m}_e\) be the vector obtained by subtracting \(1\) from every even entry of \(\mathbf{m}\), and \(\mathbf{m}_o\) the vector obtained by subtracting \(1\) from every odd entry of \(\mathbf{m}\).
\begin{enumerate}
    \item[(i)] If all entries of \(\mathbf{m}\) are at least 2,
    \[
    \facets(\CB(\mathbf{m})) = \left(\prod_{j=1}^k e_j\right)F(\mathbf{m}_e)+ \left(\prod_{j=1}^\ell o_j\right)F(\mathbf{m}_o).
    \]
    \item[(ii)] If \(o_{p+1}=\cdots=o_\ell=1\) (and \(o_p>1\)),
    \[
    \facets(\CB(\mathbf{m})) = \left(\prod_{j=1}^k e_j\right)F(\mathbf{m}_e)+ \left(\prod_{j=1}^\ell o_j\right)\facets\left(\left(\bigvee_{j=1}^k C_{e_j}\right)\vee \left(\bigvee_{j=1}^p C_{o_j-1}\right) \right).
    \]
\end{enumerate}
\end{prop}

\begin{proof}
Consider \(\CB(\mathbf{m})\) as in Figure~\ref{fig:cb gen}.
As in the proof of Proposition~\ref{prop: same parity}, we will count facets of \(P_{\CB(\mathbf{m})}\) by counting valid labelings of the facets subgraphs of \(\CB(\mathbf{m})\).
These subgraphs are those in which
\begin{enumerate}
    \item one edge of every even length path has been removed, or
    \item one edge of every odd length path has been removed.
\end{enumerate}
We can view these as labelings of \(\CB(\mathbf{m})\) where the sum of labels on each \(Q_i\) is equal, and all edges must be labeled with \(\pm 1\) except
\begin{enumerate}
    \item one edge on each even path is labeled 0, or
    \item one edge on each odd path is labeled 0. 
\end{enumerate}
In (1), there are \(\Ds\prod_{j=1}^k e_j\) ways to choose the edges to label 0.
Having the edge \(uv\) labeled \(0\) indicates that \(f(u)=f(v)\) in the corresponding facet-defining function \(f:V\rightarrow \Z\). 
Thus, we can view this edge as having been contracted since its endpoints have the same value. 
Then the reduced graph with these edges contracted is \(\CB(\mathbf{m}_e)\), constructed of paths that all have odd length. 
So, by Proposition~\ref{prop: same parity}, the number of valid labelings of \(\CB(\mathbf{m})\) where each even path has a \(0\) edge is 
\[
\left(\prod_{j=1}^k e_j\right)F(\mathbf{m}_e).
\]

In (2), there are \(\Ds\prod_{j=1}^\ell o_j\) ways to choose the edges to label 0. 
If every entry of \(\mathbf{m}\) is at least 2, the graph produced by contracting these \(0\) edges is \(\CB(\mathbf{m}_o)\), constructed of paths that all have even length. 
As above, the number of valid labelings of \(\CB(\mathbf{m})\) of this type is 
\[
\left(\prod_{j=1}^\ell o_j\right)F(\mathbf{m}_o).
\]

Thus, in case (i), 
\[
\facets(\CB(\mathbf{m}))=\left(\prod_{j=1}^k e_j\right)F(\mathbf{m}_e)+\left(\prod_{j=1}^\ell o_j\right)F(\mathbf{m}_o).
\]

To complete case (ii), note that if we contract any edge on a path of length \(1\), the endpoints of the remaining paths are identified, and the reduced graph is a wedge of cycles. 
In particular, if \(o_{p+1}=\cdots=o_\ell=1\) and \(o_p>1\), the reduced graph is \(\Ds \left(\bigvee_{j=1}^k C_{e_j}\right)\vee \left(\bigvee_{j=1}^p C_{o_j-1}\right)\). 
So the number of valid labelings of \(\CB(\mathbf{m})\) of this type is 
\[
\facets\left(\left(\bigvee_{j=1}^k C_{e_j}\right)\vee \left(\bigvee_{j=1}^p C_{o_j-1}\right)\right).
\]

Therefore, in case (ii),
\[
\facets(\CB(\mathbf{m})) = \left(\prod_{j=1}^k e_j\right)F(\mathbf{m}_e)+ \left(\prod_{j=1}^\ell o_j\right)\facets\left(\left(\bigvee_{j=1}^k C_{e_j}\right)\vee \left(\bigvee_{j=1}^p C_{o_j-1}\right) \right).
\]

\end{proof}

Returning to the special case of leafless connected graphs on $n$ vertices with $n+1$ edges, specializing to \(t=3\) provides facet counts for our graphs of interest.

\begin{corollary}\label{cor:facets of coffee beans}
The number of facets of the symmetric edge polytope for \(\CB(x_1,x_2,x_3)\) is computed as follows.
\begin{enumerate}
\item[(i)] For \(x_1, x_2, x_3\) either all even or all odd, 
\[
\facets(\CB(x_1,x_2,x_3))=F(x_1, x_2, x_3)\, .
\]

\item[(ii)] For \(o_1, o_2\) odd, \(e_1\) even, and all at least 2,
\[
\facets(\CB(o_1,o_2,e_1))=e_1F(o_1,o_2,e_1-1)+o_1o_2F(o_1-1,o_2-1,e_1)\, .
\]
For \(e_1, e_2\) even, \(o_1\) odd, and all at least 2,
\[
\facets(\CB(e_1,e_2,o_1))=o_1F(e_1,e_2,o_1-1)+e_1e_2F(e_1-1,e_2-1,o_1)\, .
\]

\item[(iii)]For \(e_1, e_2\) even, and \(o_1=1\),
\[
\facets(\CB(e_1,e_2,1))= e_1e_2F(e_1-1,e_2-1,1)+\facets(C_{e_1}\vee C_{e_2})
\]

\item[(iv)] For \(e_1\) even, \(o_1\geq 3\) odd,
\[
\facets(\CB(e_1,o_1,1))=e_1F(e_1-1,o_1,1)+o_1\facets(C_{o_1-1}\vee C_{e_1})
\]

\item[(v)] For \(e_1\) even,
\[
\facets(\CB(e_1,1,1))= e_1F(e_1-1,1,1)+\facets(C_{e_1})
\]
\end{enumerate}
\end{corollary}

\begin{example}
Figures~\ref{fig:cbexample0},~\ref{fig:cbexample1}, and~\ref{fig:cbexample2} illustrate some of the facet-defining functions for the symmetric edge polytope of \(\CB(4,2,2)\).  The vertices are labeled with their function values, and the edges are labeled ``~\(+\)~" if they are ascending and ``~\(-\)~" if they are descending.
\end{example}

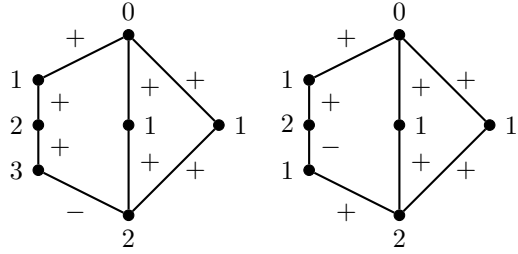
\begin{figure}

\begin{tikzpicture}
\begin{scope}[scale=.6, xshift=0, yshift=0]

	\vertex[fill,label=above:\(0\)](c1) at (-3,3) {};
	\vertex[fill, label=right:\(1\)] (c2) at (-3,1) {};
	\vertex[fill,label=below:\(2\)] (c3) at (-3,-1) {};
	\vertex[fill, label=left:\(1\)](l1) at (-5,2) {};
	\vertex[fill, label=left:\(2\)](l2) at (-5,1) {};
	\vertex[fill, label=left:\(3\)] (l3) at (-5,0) {};
	\vertex[fill, label=right:\(1\)] (r2) at (-1,1) {};
	\draw[thick] (c1)--(c2) node[pos=.6,right=0.2mm]  {\(+\)};
	\draw[thick] (c1)--(r2) node[pos=.5,right=0.2mm]  {\(+\)};
	\draw[thick] (c1)--(l1) node[pos=.6,above=0.6mm]  {\(+\)}; 
	\draw[thick] (c2)--(c3) node[pos=.4,right=0.2mm]  {\(+\)};
	\draw[thick] (r2)--(c3) node[pos=.5,right=0.2mm]  {\(+\)};
	\draw[thick] (l1)--(l2) node[pos=.5,right=0.2mm]  {\(+\)};
	\draw[thick] (l2)--(l3) node[pos=.5,right=0.2mm]  {\(+\)};
	\draw[thick] (l3)--(c3) node[pos=.4,below=0.6mm]  {\(-\)};
	\vertex[fill,label=above:\(0\)](c'1) at (3,3) {};
	\vertex[fill, label=right:\(1\)] (c'2) at (3,1) {};
	\vertex[fill,label=below:\(2\)] (c'3) at (3,-1) {};
	\vertex[fill, label=left:\(1\)](l'1) at (1,2) {};
	\vertex[fill, label=left:\(\phantom{-}2\)](l'2) at (1,1) {};
	\vertex[fill, label=left:\(1\)] (l'3) at (1,0) {};
	\vertex[fill, label=right:\(1\)] (r'2) at (5,1) {};
	\draw[thick] (c'1)--(c'2) node[pos=.6,right=0.2mm]  {\(+\)};
	\draw[thick] (c'1)--(r'2) node[pos=.5,right=0.2mm]  {\(+\)};
	\draw[thick] (c'1)--(l'1) node[pos=.6,above=0.6mm]  {\(+\)}; 
	\draw[thick] (c'2)--(c'3) node[pos=.4,right=0.2mm]  {\(+\)};
	\draw[thick] (r'2)--(c'3) node[pos=.5,right=0.2mm]  {\(+\)};
	\draw[thick] (l'1)--(l'2) node[pos=.5,right=0.2mm]  {\(+\)};
	\draw[thick] (l'2)--(l'3) node[pos=.5,right=0.2mm]  {\(-\)};
	\draw[thick] (l'3)--(c'3) node[pos=.4,below=0.6mm]  {\(+\)};
\end{scope}
\end{tikzpicture}

\caption{Some of the facet-defining functions of \(\CB(4,2,2)\) when \(j=0\)}
\label{fig:cbexample0}
\end{figure}

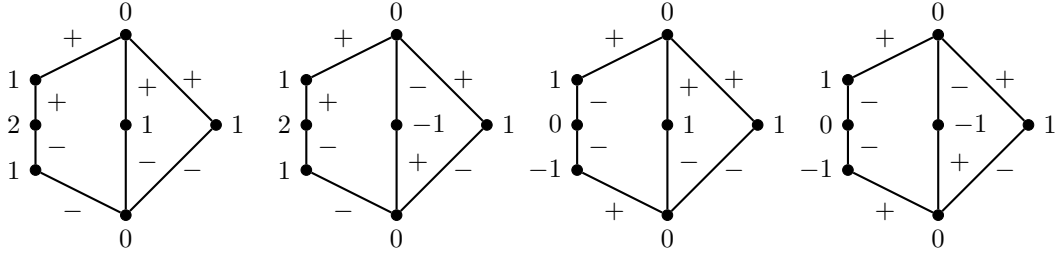
\begin{figure}
\begin{tikzpicture}
\begin{scope}[scale=.6, xshift=0, yshift=0]

	\vertex[fill,label=above:\(0\)](c1) at (-9,3) {};
	\vertex[fill, label=right:\(1\)] (c2) at (-9,1) {};
	\vertex[fill,label=below:\(0\)] (c3) at (-9,-1) {};
	\vertex[fill, label=left:\(1\)](l1) at (-11,2) {};
	\vertex[fill, label=left:\(2\)](l2) at (-11,1) {};
	\vertex[fill, label=left:\(1\)] (l3) at (-11,0) {};
	\vertex[fill, label=right:\(1\)] (r2) at (-7,1) {};
	\draw[thick] (c1)--(c2) node[pos=.6,right=0.2mm]  {\(+\)};
	\draw[thick] (c1)--(r2) node[pos=.5,right=0.2mm]  {\(+\)};
	\draw[thick] (c1)--(l1) node[pos=.6,above=0.6mm]  {\(+\)}; 
	\draw[thick] (c2)--(c3) node[pos=.4,right=0.2mm]  {\(-\)};
	\draw[thick] (r2)--(c3) node[pos=.5,right=0.2mm]  {\(-\)};
	\draw[thick] (l1)--(l2) node[pos=.5,right=0.2mm]  {\(+\)};
	\draw[thick] (l2)--(l3) node[pos=.5,right=0.2mm]  {\(-\)};
	\draw[thick] (l3)--(c3) node[pos=.4,below=0.6mm]  {\(-\)};
    \vertex[fill,label=above:\(0\)](c1) at (-3,3) {};
	\vertex[fill, label=right:\(-1\)] (c2) at (-3,1) {};
	\vertex[fill,label=below:\(0\)] (c3) at (-3,-1) {};
	\vertex[fill, label=left:\(1\)](l1) at (-5,2) {};
	\vertex[fill, label=left:\(2\)](l2) at (-5,1) {};
	\vertex[fill, label=left:\(1\)] (l3) at (-5,0) {};
	\vertex[fill, label=right:\(1\)] (r2) at (-1,1) {};
	\draw[thick] (c'1)--(c'2) node[pos=.6,right=0.2mm]  {\(+\)};
	\draw[thick] (c'1)--(r'2) node[pos=.5,right=0.2mm]  {\(+\)};
	\draw[thick] (c'1)--(l'1) node[pos=.6,above=0.6mm]  {\(+\)}; 
	\draw[thick] (c'2)--(c'3) node[pos=.4,right=0.2mm]  {\(-\)};
	\draw[thick] (r'2)--(c'3) node[pos=.5,right=0.2mm]  {\(-\)};
	\draw[thick] (l'1)--(l'2) node[pos=.5,right=0.2mm]  {\(-\)};
	\draw[thick] (l'2)--(l'3) node[pos=.5,right=0.2mm]  {\(-\)};
	\draw[thick] (l'3)--(c'3) node[pos=.4,below=0.6mm]  {\(+\)};
	\vertex[fill,label=above:\(0\)](c'1) at (3,3) {};
	\vertex[fill, label=right:\(1\)] (c'2) at (3,1) {};
	\vertex[fill,label=below:\(0\)] (c'3) at (3,-1) {};
	\vertex[fill, label=left:\(1\)](l'1) at (1,2) {};
	\vertex[fill, label=left:\(\phantom{-}0\)](l'2) at (1,1) {};
	\vertex[fill, label=left:\(-1\)] (l'3) at (1,0) {};
	\vertex[fill, label=right:\(1\)] (r'2) at (5,1) {};
	\draw[thick] (c1)--(c2) node[pos=.6,right=0.2mm]  {\(-\)};
	\draw[thick] (c1)--(r2) node[pos=.5,right=0.2mm]  {\(+\)};
	\draw[thick] (c1)--(l1) node[pos=.6,above=0.6mm]  {\(+\)}; 
	\draw[thick] (c2)--(c3) node[pos=.4,right=0.2mm]  {\(+\)};
	\draw[thick] (r2)--(c3) node[pos=.5,right=0.2mm]  {\(-\)};
	\draw[thick] (l1)--(l2) node[pos=.5,right=0.2mm]  {\(+\)};
	\draw[thick] (l2)--(l3) node[pos=.5,right=0.2mm]  {\(-\)};
	\draw[thick] (l3)--(c3) node[pos=.4,below=0.6mm]  {\(-\)};
	\vertex[fill,label=above:\(0\)](c'1) at (9,3) {};
	\vertex[fill, label=right:\(-1\)] (c'2) at (9,1) {};
	\vertex[fill,label=below:\(0\)] (c'3) at (9,-1) {};
	\vertex[fill, label=left:\(1\)](l'1) at (7,2) {};
	\vertex[fill, label=left:\(0\)](l'2) at (7,1) {};
	\vertex[fill, label=left:\(-1\)] (l'3) at (7,0) {};
	\vertex[fill, label=right:\(1\)] (r'2) at (11,1) {};
	\draw[thick] (c'1)--(c'2) node[pos=.6,right=0.2mm]  {\(-\)};
	\draw[thick] (c'1)--(r'2) node[pos=.5,right=0.2mm]  {\(+\)};
	\draw[thick] (c'1)--(l'1) node[pos=.6,above=0.6mm]  {\(+\)}; 
	\draw[thick] (c'2)--(c'3) node[pos=.4,right=0.2mm]  {\(+\)};
	\draw[thick] (r'2)--(c'3) node[pos=.5,right=0.2mm]  {\(-\)};
	\draw[thick] (l'1)--(l'2) node[pos=.5,right=0.2mm]  {\(-\)};
	\draw[thick] (l'2)--(l'3) node[pos=.5,right=0.2mm]  {\(-\)};
	\draw[thick] (l'3)--(c'3) node[pos=.4,below=0.6mm]  {\(+\)};
\end{scope}
\end{tikzpicture}

\caption{Some of the facet-defining functions of \(\CB(4,2,2)\) when \(j=1\)}
\label{fig:cbexample1}
\end{figure}

\begin{figure}
\begin{tikzpicture}
\begin{scope}[scale=.6, xshift=0, yshift=0]

	\vertex[fill,label=above:\(0\)](c1) at (-3,3) {};
	\vertex[fill, label=right:\(-1\)] (c2) at (-3,1) {};
	\vertex[fill,label=below:\(-2\)] (c3) at (-3,-1) {};
	\vertex[fill, label=left:\(-1\)](l1) at (-5,2) {};
	\vertex[fill, label=left:\(-2\)](l2) at (-5,1) {};
	\vertex[fill, label=left:\(-3\)] (l3) at (-5,0) {};
	\vertex[fill, label=right:\(-1\)] (r2) at (-1,1) {};
	\draw[thick] (c1)--(c2) node[pos=.6,right=0.2mm]  {\(-\)};
	\draw[thick] (c1)--(r2) node[pos=.5,right=0.2mm]  {\(-\)};
	\draw[thick] (c1)--(l1) node[pos=.6,above=0.6mm]  {\(-\)}; 
	\draw[thick] (c2)--(c3) node[pos=.4,right=0.2mm]  {\(-\)};
	\draw[thick] (r2)--(c3) node[pos=.5,right=0.2mm]  {\(-\)};
	\draw[thick] (l1)--(l2) node[pos=.5,right=0.2mm]  {\(-\)};
	\draw[thick] (l2)--(l3) node[pos=.5,right=0.2mm]  {\(-\)};
	\draw[thick] (l3)--(c3) node[pos=.4,below=0.6mm]  {\(+\)};
	\vertex[fill,label=above:\(0\)](c'1) at (4,3) {};
	\vertex[fill, label=right:\(-1\)] (c'2) at (4,1) {};
	\vertex[fill,label=below:\(-2\)] (c'3) at (4,-1) {};
	\vertex[fill, label=left:\(-1\)](l'1) at (2,2) {};
	\vertex[fill, label=left:\(-2\)](l'2) at (2,1) {};
	\vertex[fill, label=left:\(-1\)] (l'3) at (2,0) {};
	\vertex[fill, label=right:\(-1\)] (r'2) at (6,1) {};
	\draw[thick] (c'1)--(c'2) node[pos=.6,right=0.2mm]  {\(-\)};
	\draw[thick] (c'1)--(r'2) node[pos=.5,right=0.2mm]  {\(-\)};
	\draw[thick] (c'1)--(l'1) node[pos=.6,above=0.6mm]  {\(-\)}; 
	\draw[thick] (c'2)--(c'3) node[pos=.4,right=0.2mm]  {\(-\)};
	\draw[thick] (r'2)--(c'3) node[pos=.5,right=0.2mm]  {\(-\)};
	\draw[thick] (l'1)--(l'2) node[pos=.5,right=0.2mm]  {\(-\)};
	\draw[thick] (l'2)--(l'3) node[pos=.5,right=0.2mm]  {\(+\)};
	\draw[thick] (l'3)--(c'3) node[pos=.4,below=0.6mm]  {\(-\)};
\end{scope}
\end{tikzpicture}

\caption{Some of the facet-defining functions of \(\CB(4,2,2)\) when \(j=2\)}
\label{fig:cbexample2}
\end{figure}
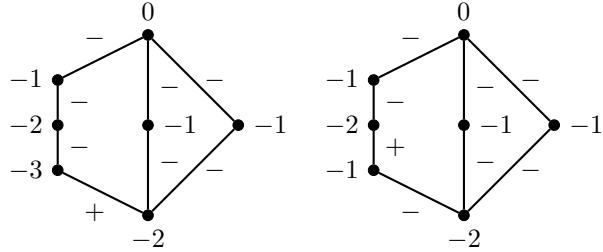

Using these results, we can make partial progress toward Conjecture~\ref{conj:nn+1} in two special cases, given below in Theorem~\ref{thm:general F less than M} and Proposition~\ref{prop:max CB less than M}.

\begin{theorem}{\label{thm:general F less than M}}
For all \(n\geq 4\), if \(x_1\geq x_2\geq x_3\geq 1\), all \(x_i\)'s have the same parity, and \(x_1+x_2+x_3=n+1\), then
\[
F(x_1,x_2,x_3)\leq M(n).
\]
Thus, if \(x_1\), \(x_2\), \(x_3\) are all of the same parity, then Conjecture~\ref{conj:nn+1} is true.
\end{theorem}

\begin{remark}
The proof of Theorem~\ref{thm:general F less than M} makes use of the Stirling bounds on \(n!\) given in~\cite{robbins1955stirling}. 
Namely,
\[
\sqrt{2\pi}\phantom{\cdot}n^{n+\frac{1}{2}}e^{-n}e^\frac{1}{12n+1}\leq n!\leq \sqrt{2\pi}\phantom{\cdot}n^{n+\frac{1}{2}}e^{-n}e^\frac{1}{12n}.
\]
\end{remark}

\begin{proof}[Proof of Theorem~\ref{thm:general F less than M}]
In each of the following cases, we show that the desired inequality holds for large enough \(n\).
For all smaller values of \(n\) we have verified that the theorem holds using SageMath~\cite{sage}. 
Throughout the proof, we use the notation \(\overset{!}{\leq}\) to indicate an unproven inequality we wish to show.

\noindent\emph{Case 1 (\(n=2k-1\)):} In this case, \(x_1+x_2+x_3=2k\), and all \(x_i\)'s are even by assumption.
Then we have:
\[
\begin{split}
    F(x_1,x_2,x_3)&=\sum_{j=0}^{x_3}\binom{x_3}{j}\binom{x_2}{\frac{1}{2}\left(x_2-x_3\right)+j}\binom{x_1}{\frac{1}{2}\left(x_1-x_3\right)+j}\\&\leq (x_3+1)\binom{x_3}{\frac{x_3}{2}}\binom{x_2}{\frac{x_2}{2}}\binom{x_1}{\frac{x_1}{2}}
\end{split}
\]

\emph{Subcase 1(a):} If \(k\) is even, 
\[
M(n)=M(2k-1)=(k+1)(k-1)\binom{k}{\frac{k}{2}}\binom{k-2}{\frac{k-2}{2}}.
\]
In this case, to show \(F(x_1,x_2,x_3)\leq M(n)\), it suffices to show
\begin{equation}\label{ineq: WTS case 1a}
\begin{split}
    (x_3+1)\phantom{\cdot}x_3!x_2!x_1!\left(\frac{k}{2}!\right)^2\left(\frac{k-2}{2}!\right)^2\overset{!}{\leq} (k+1)(k-1)k!(k-2)!\left(\frac{x_3}{2}!\right)^2\left(\frac{x_2}{2}!\right)^2\left(\frac{x_1}{2}!\right)^2
    \end{split}
\end{equation}
By the Stirling bounds on \(n!\), it suffices to show

\[
\begin{pmatrix}
(x_3+1)x_3^{x_3+\frac{1}{2}}x_2^{x_2+\frac{1}{2}}x_1^{x_1+\frac{1}{2}} &
\\&
\\\cdot \left(\frac{k}{2}\right)^{k+1}\left(\frac{k-2}{2}\right)^{k-1} &
\\&
\\ \cdot e^{-(x_1+x_2+x_3+2k)+2} &
\\&
\\ \cdot e^{\frac{1}{12x_3}+\frac{1}{12x_2}+\frac{1}{12x_1}+\frac{1}{3k}+\frac{1}{3(k-2)}} &
\end{pmatrix}
\overset{!}{\leq} 
\begin{pmatrix}
&\sqrt{2\pi}\left(\frac{x_3}{2}\right)^{x_3+1}\left(\frac{x_2}{2}\right)^{x_2+1}\left(\frac{x_1}{2}\right)^{x_1+1}
\\&
\\&\cdot (k+1)(k-1)k^{k+\frac{1}{2}}(k-2)^{k-\frac{3}{2}}
\\&
\\&\cdot e^{-(x_1+x_2+x_3+2k)+2}
\\&
\\&\cdot e^{\frac{1}{12k+1}+\frac{1}{12k-23}+ \frac{2}{6x_3+1}+ \frac{2}{6x_2+1}+ \frac{2}{6x_1+1}}
\end{pmatrix}
\]

or equivalently,
\[
\begin{pmatrix}
    (x_3+1)&
    \\&
    \\\cdot k^{k+1}(k-2)^{k-1}&
    \\&
    \\\cdot e^{\frac{1}{12x_3}+\frac{1}{12x_2}+\frac{1}{12x_1}+\frac{1}{3k}+\frac{1}{3(k-2)}}&
\end{pmatrix}
    \overset{!}{\leq}
\begin{pmatrix}
    &\frac{\sqrt{2\pi}}{8}\sqrt{x_1x_2x_3}
    \\&
    \\&\cdot (k+1)(k-1)k^{k+\frac{1}{2}}(k-2)^{k-\frac{3}{2}}
    \\&
    \\&\cdot e^{\frac{1}{12k+1}+\frac{1}{12k-23}+ \frac{2}{6x_3+1}+ \frac{2}{6x_2+1}+ \frac{2}{6x_1+1}}.
\end{pmatrix}
\]

Since \(\frac{1}{12k+1}+\frac{1}{12k-23}> 0\) for \(k\geq 2\), 
\[
e^{\frac{1}{12k+1}+\frac{1}{12k-23}} > 1.
\]
Also, 
\[
-1\leq \frac{1}{12x}-\frac{2}{6x+1} \leq 0
\]
for all \(x\geq 1\) and so
\[
0\leq e^{\frac{1}{12x_3}-\frac{2}{6x_3+1}+\frac{1}{12x_2}-\frac{2}{6x_2+1}+\frac{1}{12x_1}-\frac{2}{6x_1+1}}\leq 1
\]
for all \(x_1,x_2,x_3\).
Therefore, to show inequality~\eqref{ineq: WTS case 1a}, it suffices to show 
\[
 (x_3+1)k^{k+1}(k-2)^{k-1}e^{\frac{1}{3k}+\frac{1}{3(k-2)}}\overset{!}{\leq}\frac{\sqrt{2\pi}}{8}(k+1)(k-1)k^{k+\frac{1}{2}}(k-2)^{k-\frac{3}{2}}\sqrt{x_1x_2x_3}
\]
or rather,

\[
 (x_3+1)\sqrt{k(k-2)}e^{\frac{1}{3k}+\frac{1}{3(k-2)}}\overset{!}{\leq}\frac{\sqrt{2\pi}}{8}(k+1)(k-1)\sqrt{x_1x_2x_3}
\]

Finally, we note the following:
\begin{itemize}
    \item By assumption, \(\Ds x_3+1\leq\frac{2k}{3}+1\leq k+1\), and so \(\Ds\frac{x_3+1}{k+1}\leq 1\).
    \item \(\Ds\frac{\sqrt{k(k-2)}}{k-1}\leq 1\).
    \item \(\Ds 0< e^{\frac{1}{3k}+\frac{1}{3(k-2)}} < e\) for \(k\geq 3\).
    \item By assumption, \(\Ds x_1\geq \frac{2k}{3}\), and \(x_2,x_3\geq 2\), implying \(\Ds x_1x_2x_3\geq \frac{8k}{3}\). 
\end{itemize}

With this, it suffices to show
\[
e\overset{!}{\leq} \frac{\sqrt{2\pi}}{8}\sqrt{\frac{8k}{3}}
\]
or 
\[
k\overset{!}{\geq} \frac{12e^2}{\pi}\approx 28.224.
\]
This inequality and the desired inequality hold for all even \(k\geq 30\).

\emph{Subcase 1(b):} If \(k\) is odd,
\[
M(n)=M(2k-1)=k^2\binom{k-1}{\frac{k-1}{2}}.
\]

To show \(F(x_1,x_2,x_3)\leq M(n)\), it suffices to show
\begin{equation}\label{ineq: WTS case 1b}
\begin{split}
    (x_3+1)\phantom{\cdot}x_3!x_2!x_1!\left(\frac{k-1}{2}!\right)^4\overset{!}{\leq} (k!)^2\left(\frac{x_3}{2}!\right)^2\left(\frac{x_2}{2}!\right)^2\left(\frac{x_1}{2}!\right)^2.
    \end{split}
\end{equation}

By the Stirling bounds on \(n!\), it suffices to show
\[
\begin{pmatrix}
(x_3+1)x_3^{x_3+\frac{1}{2}}x_2^{x_2+\frac{1}{2}}x_1^{x_1+\frac{1}{2}}&
\\&
\\\cdot\left(\frac{k-1}{2}\right)^{2k}&
\\&
\\\cdot e^{-(x_1+x_2+x_3+2k)+2}&
\\&
\\\cdot e^{\frac{1}{12x_3}+\frac{1}{12x_2}+\frac{1}{12x_1}+\frac{2}{3(k-1)}}
\end{pmatrix}
    \overset{!}{\leq} 
\begin{pmatrix}
&\sqrt{2\pi}\left(\frac{x_3}{2}\right)^{x_3+1}\left(\frac{x_2}{2}\right)^{x_2+1}\left(\frac{x_1}{2}\right)^{x_1+1}
\\&
\\&\cdot k^{2k+1}
\\&
\\&\cdot e^{-(x_1+x_2+x_3+2k)}
\\&
\\&\cdot e^{\frac{2}{12k+1}+ \frac{2}{6x_3+1}+ \frac{2}{6x_2+1}+ \frac{2}{6x_1+1}}.
\end{pmatrix}
\]
Using the same kinds of computations as the previous case, we see it suffices to show
\[
e^2(x_3+1)(k-1)^{2k}e^{\frac{2}{3(k-1)}}\overset{!}{\leq} \frac{\sqrt{2\pi}}{8}k^{2k+1}\sqrt{x_1x_2x_3}
\]

Now note that:
\begin{itemize}
    \item \(\Ds x_3+1\leq \frac{2k}{3}+1\leq k\) for \(k\geq 3\), and so \(\frac{x_3+1}{k}\leq 1\).
    \item \(\Ds e^2\left(\frac{k-1}{k}\right)^{2k}\leq 1\) for \(k\geq 3\).
    \item \(1\leq e^{\frac{2}{3(k-1)}}\leq e\) for \(k\geq 2\).
\end{itemize}
So, it suffices to show
\[
e\overset{!}{\leq}\frac{\sqrt{2\pi}}{8}\sqrt{x_1x_2x_3}, 
\]
which, as before, holds for 
\[
k\geq \frac{12e^2}{\pi}\approx 28.224 
\]
or all odd \(k\geq 29\).

\noindent\emph{Case 2 (\(n=2k\)):} In this case, \(x_1+x_2+x_3=2k+1\), and all \(x_i\)'s are odd by assumption.
Then,
\[
\begin{split}
    F(x_1,x_2,x_3)&=\sum_{j=0}^{x_3}\binom{x_3}{j}\binom{x_2}{\frac{1}{2}\left(x_2-x_3\right)+j}\binom{x_1}{\frac{1}{2}\left(x_1-x_3\right)+j}\\&\leq (x_3+1)\binom{x_3}{\frac{x_3-1}{2}}\binom{x_2}{\frac{x_2-1}{2}}\binom{x_1}{\frac{x_1-1}{2}}
\end{split}
\]

\emph{Subcase 2(a):} If \(k\) is even, 
\[
M(n)=M(2k)=2(k+1)(k-1)\binom{k}{\frac{k}{2}}\binom{k-2}{\frac{k-2}{2}}.
\]
To show \(F(x_1,x_2,x_3)\leq M(n)\), it suffices to show
\begin{equation}\label{ineq: WTS case 2a}
\begin{split}
    &(x_3+1)\phantom{\cdot}x_3!x_2!x_1!\left(\frac{k}{2}!\right)^2\left(\frac{k-2}{2}!\right)^2\\&\overset{!}{\leq} 2(k+1)(k-1)k!(k-2)!\left(\frac{x_3-1}{2}!\right)\left(\frac{x_3+1}{2}!\right)\left(\frac{x_2-1}{2}!\right)\left(\frac{x_2+1}{2}!\right)\left(\frac{x_1-1}{2}!\right)\left(\frac{x_1+1}{2}!\right).
    \end{split}
\end{equation}
Equivalently,
\[
\begin{split}
&(x_3+1)\phantom{\cdot}x_3!x_2!x_1!\left(\frac{k}{2}!\right)^2\left(\frac{k-2}{2}!\right)^2\\&\overset{!}{\leq} 2(k+1)(k-1)k!(k-2)!\left(\frac{x_3+1}{2}!\right)^2\left(\frac{x_2+1}{2}!\right)^2\left(\frac{x_1+1}{2}!\right)^2\left(\frac{8}{(x_3+1)(x_2+1)(x_1+1)}\right),
\end{split}
\]
or
\[
\begin{split}
&(x_3+1)\phantom{\cdot}(x_3+1)!(x_2+1)!(x_1+1)!\left(\frac{k}{2}!\right)^2\left(\frac{k-2}{2}!\right)^2\\&\overset{!}{\leq} 16(k+1)(k-1)k!(k-2)!\left(\frac{x_3+1}{2}!\right)^2\left(\frac{x_2+1}{2}!\right)^2\left(\frac{x_1+1}{2}!\right)^2.
\end{split}
\]

By the Stirling bounds, it suffices to show
\[
\begin{pmatrix}
    (x_3+1)(x_3+1)^{x_3+\frac{3}{2}}(x_2+1)^{x_2+\frac{3}{2}}(x_1+1)^{x_1+\frac{3}{2}}&
    \\&
    \\\cdot \left(\frac{k}{2}\right)^{k+1}\left(\frac{k-2}{2}\right)^{k-1}&
    \\&
    \\\cdot e^{-(x_3+x_2+x_1+2k+1)}&
    \\&
    \\\cdot e^{\frac{1}{12(x_3+1)}+\frac{1}{12(x_2+1)}+\frac{1}{12(x_1+1)}+\frac{1}{3k}+\frac{1}{3(k-2)}}&
\end{pmatrix}    
    \overset{!}{\leq}
\begin{pmatrix}
    & 16\sqrt{2\pi}\left(\frac{x_3+1}{2}\right)^{x_3+2}\left(\frac{x_2+1}{2}\right)^{x_2+2}\left(\frac{x_1+1}{2}\right)^{x_1+2}
    \\&
    \\&\cdot (k+1)(k-1)k^{k+\frac{1}{2}}(k-1)^{k-\frac{3}{2}}
    \\&
    \\&\cdot e^{-(x_3+x_2+x_1+2k+1)}
    \\&
    \\&\cdot e^{\frac{1}{12k+1}+\frac{1}{12(k-2)+1}+\frac{2}{6(x_3+1)+1}+\frac{2}{6(x_2+1)+1}+\frac{2}{6(x_1+1)+1}}
\end{pmatrix}
\]
After computations similar to those in Case 1, we see it suffices to show
\[
(x_3+1)\sqrt{k(k-2)}e\overset{!}{\leq} \frac{\sqrt{2\pi}}{8}(k+1)(k-1)\sqrt{(x_1+1)(x_2+1)(x_3+1)}
\]
where \(x_1+1\geq \frac{2k+4}{3}\), \(x_2+1\geq 2\), \(x_3+1\geq 2\).
It suffices to have 
\[
e\overset{!}{\leq}\frac{\sqrt{2\pi}}{8}\sqrt{4\left(\frac{2k+4}{3}\right)}
\]
or 
\[
k\geq \frac{12e^2}{\pi}-2\approx 26.224.
\]
So the desired inequality holds for even \(k\geq 28\).

\emph{Subcase 2(b):} If \(k\) is odd,

\[
M(n)=M(2k)=2k^2\binom{k-1}{\frac{k-1}{2}}^2
\]

To show \(F(x_1,x_2,x_3)\leq M(n)\) it suffices to show

\begin{equation}
\begin{split}
    &(x_3+1)x_3!x_2!x_1!\left(\frac{k-1}{2}!\right)^4\\&\overset{!}{\leq} 2(k!)^2\left(\frac{x_3-1}{2}!\right)\left(\frac{x_3+1}{2}!\right)\left(\frac{x_2-1}{2}!\right)\left(\frac{x_2+1}{2}!\right)\left(\frac{x_1-1}{2}!\right)\left(\frac{x_1+1}{2}!\right).
    \end{split}
\end{equation}

Equivalently,

\[
\begin{split}
 &(x_3+1)\phantom{\cdot}x_3!x_2!x_1!\left(\frac{k-1}{2}!\right)^4
 \\&\overset{!}{\leq} 2(k!)^2\left(\frac{x_3+1}{2}!\right)^2\left(\frac{x_2+1}{2}!\right)^2\left(\frac{x_1+1}{2}!\right)^2\left(\frac{8}{(x_3+1)(x_2+1)(x_1+1)}\right),   
\end{split}
\]

or

\[
\begin{split}
    &(x_3+1)(x_3+1)!(x_2+1)!(x_1+1)!\left(\frac{k-1}{2}!\right)^4
    \\&\overset{!}{\leq} 16(k!)^2\left(\frac{x_3+1}{2}!\right)^2\left(\frac{x_2+1}{2}!\right)^2\left(\frac{x_1+1}{2}!\right)^2
\end{split}
\]

By the Stirling bounds, it suffices to show
\[
\begin{pmatrix}
    (x_3+1)(x_3+1)^{x_3+\frac{3}{2}}(x_2+1)^{x_2+\frac{3}{2}}(x_1+1)^{x_1+\frac{3}{2}}&
    \\&
    \\\cdot\left(\frac{k-1}{2}\right)^{2k}&
    \\&
    \\\cdot e^{-(x_3+x_2+x_1+2k+1)}
    \\&
    \\\cdot e^{\frac{1}{12(x_3+1)}+\frac{1}{12(x_2+1)}+\frac{1}{12(x_1+1)}+\frac{2}{3(k-1)}}
\end{pmatrix}    
    \overset{!}{\leq} 
\begin{pmatrix}    
    &16\sqrt{2\pi}\phantom{\cdot}\left(\frac{x_3+1}{2}\right)^{x_3+2}\left(\frac{x_2+1}{2}\right)^{x_2+2}\left(\frac{x_1+1}{2}\right)^{x_1+2}
    \\&
    \\&\cdot k^{2k+1}
    \\&
    \\&\cdot e^{-(x_3+x_2+x_1+2k+3)}
    \\&
    \\&\cdot e^{\frac{2}{12k+1}+\frac{2}{6(x_3+1)+1}+\frac{2}{6(x_2+1)+1}+\frac{2}{6(x_1+1)+1}}
\end{pmatrix}
\]

After computations similar to those in previous cases, we see it suffices to show
\[
e^2(x_3+1)(k-1)^{2k}e^{\frac{2}{3(k-1)}}\overset{!}{\leq}\frac{\sqrt{2\pi}}{8}k^{2k+1}\sqrt{(x_1+1)(x_2+1)(x_3+1)} \, .
\]

As in Case 1(b), it suffices to show 
\[
e\overset{!}{\leq}\frac{\sqrt{2\pi}}{8}\sqrt{(x_1+1)(x_2+1)(x_3+1)}
\]
with \(x_1+1\geq \frac{2k+4}{3}\), \(x_2+1\geq 2\), \(x_3+1\geq 2\).
The desired inequality holds for 
\[
k\geq \frac{12e^2}{\pi}-2\approx 26.224
\]
or all odd \(k\geq 27\).
\end{proof}

Our second special case concerns a certain family of \(\CB\) graphs with an even number of vertices where the two cycles share exactly one edge.
We will need the following lemma, the proof of which follows from straightforward computations after expanding the right hand sides. 

\begin{lemma}\label{lem:auxeqs}
\noindent If \(k\) is even,
    \begin{equation}\label{eqn:even M and F equality first}
    M(2k)=\left(\frac{k+2}{2}\right)\left(\frac{k}{2}\right)F(k+1,k-1,1).
    \end{equation}
If \(k\) is odd,
    \begin{equation}\label{eqn:even M and F equality second}
    M(2k)=\left(\frac{k+1}{2}\right)^2F(k,k,1).
    \end{equation}
For even \(k\)
\begin{equation}\label{eqn: M incrementing equality even}
M(2k-2)=\frac{k}{2(k+1)}M(2k-1).
\end{equation} 
For odd \(k\)
\begin{equation}\label{eqn: M incrementing equality odd}
M(2k-2)=\frac{k-1}{2k}M(2k-1).
\end{equation}
For all \(k\)
\begin{equation}\label{eqn: M incrementing equality all}
M(2k-1)=\frac{1}{2}M(2k).
\end{equation}
Finally, if \(k\) is even, 
\begin{equation}\label{eqn: N(cycle) equality}
\facets(C_k)=\frac{4}{k}\facets(C_{k-1}) \, .
\end{equation}
\end{lemma}

\begin{prop}\label{prop:max CB less than M}
Let \(n=2k\geq 10\).
\begin{enumerate}
    \item[(i)] If \(k\) is even,
    \[
    \facets(\CB(k,k,1))\leq M(2k).
    \]
    \item[(ii)] If \(k\) is odd,
    \[
    \facets(\CB(k+1,k-1,1))\leq M(2k).
    \]
\end{enumerate}
\end{prop}

\begin{proof}
For even \(k\), we have
\[
\begin{split}
    \facets(\CB(k,k,1))&=\facets(C_k\vee C_k)+k^2F(k-1,k-1,1)
    \\&\overset{\eqref{eqn: N(cycle) equality}, \eqref{eqn:even M and F equality second}}{=}\frac{16}{k^2}\facets(C_{k-1}\vee C_{k-1})+\frac{4k^2}{k^2}M(2k-2)
    \\&\overset{\mathrm{Def}~\ref{def:M(n)}}{=}\frac{16}{k^2}M(2k-3)+4M(2k-2)
    \\&\overset{\eqref{eqn: M incrementing equality all},\eqref{eqn: M incrementing equality even}}{=}\frac{2}{k(k+1)}M(2k)+\frac{k-1}{k}M(2k)
    \\&=\frac{k^2+1}{k^2+k}M(2k)\leq M(2k).
\end{split}
\]

For odd \(k\),
\[
\begin{split}
    \facets(\CB(k+1,k-1,1))&=\facets(C_{k+1}\vee C_{k-1})+k^2F(k-1,k-1,1)
    \\&\overset{\eqref{eqn: N(cycle) equality},\, \eqref{eqn:even M and F equality first}}{=}\frac{16}{(k+1)(k-1)}\facets(C_{k}\vee C_{k-2})+\frac{4(k+1)(k-1)}{(k+1)(k-1)}M(2k-2)
    \\&\overset{\mathrm{Def}~\ref{def:M(n)}}{=}\frac{16}{(k+1)(k-1)}M(2k-3)+4M(2k-2)
    \\&\overset{\eqref{eqn: M incrementing equality all},\eqref{eqn: M incrementing equality odd}}{=}\frac{2}{k(k+1)}M(2k)+\frac{k}{k+1}M(2k)
    \\&=\frac{k^2+2}{k^2+k}M(2k)\leq M(2k).
\end{split}
\]
\end{proof}

\section{Further Conjectures and Open Problems}\label{sec:conjectures}

Through the course of this study, we observed several patterns that remain as conjectures and open questions.
First, computational evidence suggests interesting structure for the function \(F(x_1,x_2,x_3)\) beyond Theorem~\ref{thm:general F less than M}.
We formally record our observations as the following conjecture, which has been confirmed with SageMath~\cite{sage} for all $n$ less than or equal to $399$.

\begin{conjecture}\label{conj:Fbounds}
For \(n=2k\) and \(k\geq 2\) with \(x_1+x_2+x_3=n+1\), \(F(x_1,x_2,x_3)\) is maximized at \(F(n-1,1,1)\).
For \(n=2k-1\) and \(k\geq 3\) with \(x_1+x_2+x_3=n+1\), \(F(x_1,x_2,x_3)\) is maximized at \(F(n-3,2,2)\).
Further, for any \(x_1\geq x_2\geq x_3\geq 3\) all even or all odd positive integers,
\[
F(x_1,x_2,x_3)\leq F(x_1+2,x_2,x_3-2)
\]
and
\[
F(x_1,x_2,x_3)\leq F(x_1+2,x_2-2,x_3) \, ,
\]
when the subtraction by \(2\) will maintain the inequalities on the \(x_i\)'s.
\end{conjecture}

For example, the first inequality in Conjecture~\ref{conj:Fbounds} asserts that for $x_1\geq x_2 \geq x_3\geq 5$ all of the same parity,
\begin{align*}
& \sum_{j=0}^{x_3}\binom{x_3}{j}\binom{x_2}{\frac{1}{2}\left(x_2-x_3\right)+j}\binom{x_1}{\frac{1}{2}\left(x_1-x_3\right)+j}\\
\leq & \sum_{j=0}^{x_3-2}\binom{x_3-2}{j}\binom{x_2}{\frac{1}{2}\left(x_2-x_3+2\right)+j}\binom{x_1+2}{\frac{1}{2}\left(x_1-x_3\right)+j} \, .
\end{align*}

Second, the remaining case for Conjecture~\ref{conj:nn+1} is the following.

\begin{conjecture}\label{conj:remainingcb}
If \(x_1\), \(x_2\), and \(x_3
\) have different parities, then \(N(\CB(x_1,x_2,x_3))\leq M(n)\).
\end{conjecture}

Using the recursion given by Corollary~\ref{cor:facets of coffee beans} part (ii) and the inequality \(x_i\leq n\), it is straightforward to deduce that \(N(\CB(x_1,x_2,x_3))\leq 6n^2M(n)\).
It is not clear to the authors how to obtain a stronger bound in this case.
One direction toward proving Conjecture~\ref{conj:remainingcb} is the following.

\begin{conjecture}\label{conj: Mixed parity CB maximizers}
For \(n\geq 10\), \(\facets(\CB(x_1,x_2,x_3))\) with \(x_1+x_2+x_3=n+1\) is maximized by
\[
\begin{cases}
\CB(k-1,k-1,2)\; &n=2k-1, \text{\(k\) even}\\
\CB(k,k-2,2)\; &n=2k-1, \text{\(k\) odd}\\
\CB(k,k,1)\; &n=2k, \text{\(k\) even}\\
\CB(k+1,k-1,1)\; &n=2k, \text{\(k\) odd}\\
\end{cases}.
\]
\end{conjecture}

Using SageMath~\cite{sage}, we have computed \(\facets(\CB(x_1,x_2,x_3))\) for all tuples with \(x_1+x_2+x_3=n+1\leq 535\). 
All of these values are less than or equal to the number of facets of our conjectured maximizer for the corresponding \(n\), providing significant support for Conjecture~\ref{conj: Mixed parity CB maximizers}.

Third, when \(n\) is even, Proposition~\ref{prop:max CB less than M} gives that the number of facets given by these conjectured maximizing graphs remains less than \(M(n)\). 
Currently, for odd \(n\) we do not know of an equality or a bound strong enough to accomplish what~\eqref{eqn:even M and F equality first}~and~\eqref{eqn:even M and F equality second} give for even \(n\). 
Therefore, a similar result for odd \(n\) remains unproven.
We have verified that such a result holds for all odd \(n\) less than 100,000 via computations with SageMath~\cite{sage}.

Fourth and finally, throughout our investigations we sought examples of graphs having a high number of symmetric edge polytope facets. Conjecture~\ref{conj:maxmin} asserts that graphs appearing as global facet-maximizers for connected graphs on $n$ vertices can be constructed from minimally intersecting odd cycles, but it is unclear how to prove this.
A related problem would be to prove that the graphs appearing as global facet-maximizers in Conjecture~\ref{conj:maxmin} are facet-maximizers among connected graphs having a fixed number of edges.
We explore this idea a bit further in the special case of the following graphs, which are the conjectured global facet-maximizers for connected graphs on an odd number of vertices.

\begin{definition}
Let \(\windmill(n,r)\) denote the \emph{windmill} graph on \(n\) vertices consisting of \(r\) copies of the cycle \(C_3\) and \(n-1-2r\) edges all wedged at a single vertex.
We say a windmill is \emph{full} if \(n\) is odd and \(r=\frac{n-1}{2}\).
In other words, a full windmill is a wedge of \(\frac{n-1}{2}\) triangles at a single vertex.
Denote by \(\windmill(n)\) the full windmill on \(n\) vertices.
\end{definition}

 \begin{figure}
\begin{center}
    \begin{tikzpicture}
    
    \begin{scope}[scale=1, xshift=0, yshift=0]
	\vertex[fill](v1) at (0,0) {};
	\vertex[fill](v2) at (0,1.5) {};
	\vertex[fill](v3) at (1.25,.75) {};
	\vertex[fill](v4) at (1.25,-.75) {};
	\vertex[fill](v5) at (0,-1.5) {};	
	\vertex[fill](v6) at (-1.25,-.75) {};
	\vertex[fill](v7) at (-1.25,.75) {};	
	\draw[thick] (v1)--(v2);
	\draw[thick] (v1)--(v3);
	\draw[thick] (v1)--(v4);
	\draw[thick] (v1)--(v5);	
	\draw[thick] (v1)--(v6);
	\draw[thick] (v1)--(v7);
	\draw[thick] (v2)--(v3);
	\draw[thick] (v4)--(v5);
    \node[] at (-2.5,1) {\(\windmill(7,2)\)};   
    \end{scope}

    \begin{scope}[scale=1, xshift=200, yshift=0]
	\vertex[fill](v1) at (0,0) {};
	\vertex[fill](v2) at (0,1.5) {};
	\vertex[fill](v3) at (1.25,.8) {};
	\vertex[fill](v4) at (1.25,-.3) {};
	\vertex[fill](v5) at (.55,-1.35) {};	
	\vertex[fill](v6) at (-1.25,-.5) {};
	\vertex[fill](v7) at (-1.25,.8) {};
	\vertex[fill](v8) at (-0.6, -1.35) {};
	\draw[thick] (v1)--(v2);
	\draw[thick] (v1)--(v3);
	\draw[thick] (v1)--(v4);
	\draw[thick] (v1)--(v5);	
	\draw[thick] (v1)--(v6);
	\draw[thick] (v1)--(v7);
	\draw[thick] (v1)--(v8);
 	\draw[thick] (v2)--(v3);
 	\draw[thick] (v4)--(v5);
    \node[] at (-2.5,1) {\(\windmill(8,2)\)}; 
    \end{scope}    
    \end{tikzpicture}    
\end{center}
    \caption{Two windmill graphs which are not full.}
    \label{fig:windmills}
\end{figure}
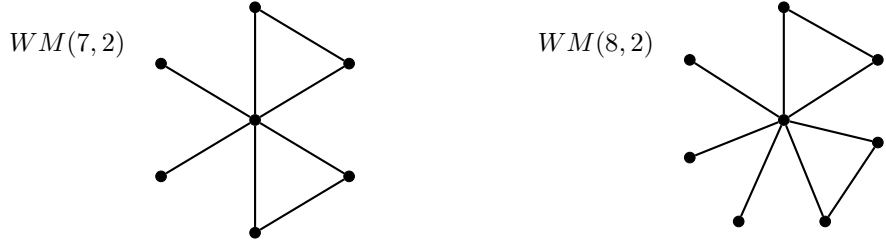

\begin{prop}
For all odd \(n\), 
\[
\facets(\windmill(n))=6^{\frac{n-1}{2}}
\]
\end{prop}

\begin{proof}
The windmill \(\windmill(n)\) is a join of \(\frac{n-1}{2}\) 3-cycles.  By Lemma~\ref{lem:cyclefacets} and Proposition~\ref{prop:wedgesmultiply},
\[
\facets(\windmill(n))=\left(\facets(C_3)\right)^{\frac{n-1}{2}}=6^{\frac{n-1}{2}}.
\]
\end{proof}

\begin{conjecture}\label{conj:windmill}
Among graphs with \(n\) vertices and \(3(n-1)/2\) edges (for odd \(n\)), \windmill(n) is a facet-maximizer.
\end{conjecture}

To support this conjecture, we used SageMath~\cite{sage} to sample the space of connected graphs with \(n\) vertices and \(3(n-1)/2\) edges using a Markov Chain Monte Carlo technique~\cite[Section~2]{SIAMMCMC}.
Then we computed \(\facets(\PG)\) for each graph \(G\) in our sample.
The transition operation we consider is an edge replacement.
Starting at a graph \(G\), we produce a new graph \(G'\) by randomly choosing an edge \(e\in E(G)\) and a non-edge \(f\in E(G)^C\).
Then, if the edges \((E(G)\setminus\{e\})\cup\{f\}\) form a connected graph, define \(G'\) to be this graph. 
If the new graph is not connected, let \(G'=G\) (in other words, sample at \(G\) again).

Using this single-edge replacement, the resulting graph of graphs \(\mathcal{G}\) is regular~\cite{SIAMMCMC}, with each node having in-degree and out-degree both equal to
\[\frac{3}{2}(n-1)\left(\binom{n}{2}-\frac{3}{2}(n-1)\right)\, .
\]
Given any two graphs \(G_1\) and \(G_2\) in the space, there is a sequence of edge replacements that first transforms a spanning tree of \(G_1\) into a spanning tree of \(G_2\) and then replaces all other edges in \(E(G_1)\setminus E(G_2)\) with edges in \(E(G_2)\setminus E(G_1)\) in any order.  
Thus \(G_2\) is reachable from \(G_1\), and, since all edge replacements are reversible, \(G_1\) is reachable from \(G_2\).
Thus \(\mathcal{G}\) is strongly connected.
Finally, it is straightforward to see that \(\mathcal{G}\) is aperiodic, as it contains 2-cycles and 3-cycles.
Thus, we can conclude that samples from this Markov chain asymptotically obey a uniform distribution, and we can assume that this process uniformly samples the space of connected graphs with \(n\) vertices and \(3(n-1)/2\) edges.
We generated sample families of graphs for all odd \(n\) between 5 and 17.
The results of our sampling, shown in Figures~\ref{fig:allWM}~and~\ref{fig:WM13}, support Conjecture~\ref{conj:windmill} for these values of \(n\).

\begin{figure}
    \centering
    \includegraphics[width=10cm]{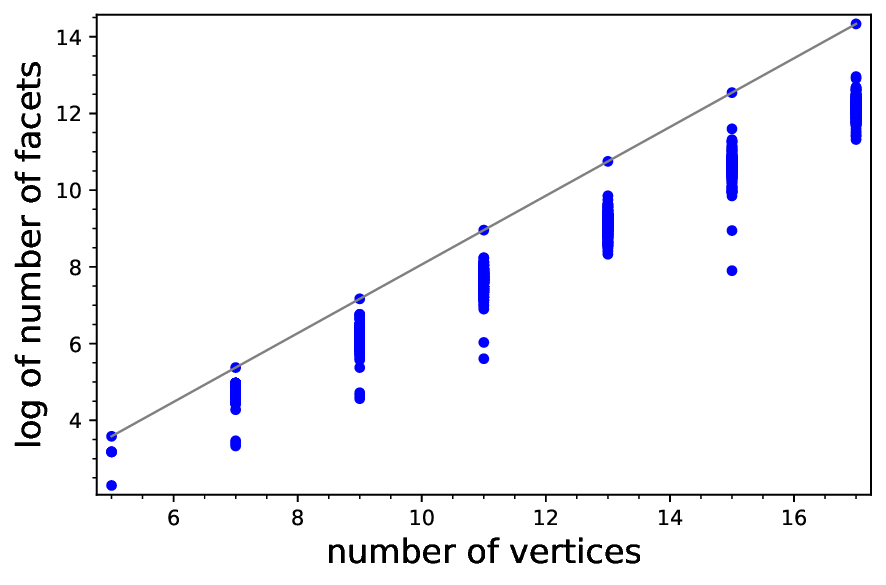}
    \caption{For each odd \(n\) between 5 and 17, the plot shows the log of the number of facets of \(\PG\) for samples of graphs \(G\) with \(n\) vertices and \(3(n-1)/2\) edges with a target sample size of 200 graphs for each \(n\). The line is \( y=\frac{\log(6)}{2}(x-1)\), indicating \(\facets(\windmill(n))\) for each \(n\).}
    \label{fig:allWM}
\end{figure}

\begin{figure}
    \centering
    \includegraphics[width=10cm]{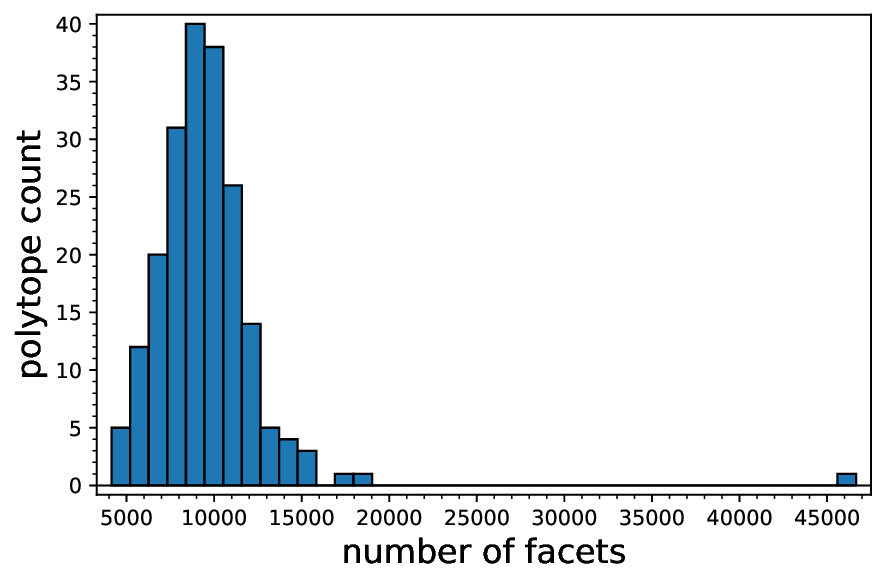}
    \caption{For \(n=13\), the histogram shows the distribution of \(\facets(\PG)\) for our sample graphs. Not only does the maximum number of facets in our sample occur at \(6^6=\facets(\windmill(13))\), but there is a significant gap between our maximizer and all other facet counts in our sample.}
    \label{fig:WM13}
\end{figure}

The complexity of counting facets and determining which graphs are facet-maximizers in a case as small as graphs with \(n\) vertices and \(n+1\) edges was unexpected and indicates that there are many factors at play.
Therefore, counting the facets of symmetric edge polytopes remains an interesting problem in terms of both establishing formulas and investigating new techniques.

\bibliographystyle{plain}
\bibliography{Braun}

\end{document}